\documentclass[a4paper]{article}
\usepackage{amsmath,amssymb}
\usepackage[dvipdfmx]{graphicx}
\usepackage{comment}
\usepackage[all]{xy}

\newtheorem{theorem}{Theorem}[section]
\newtheorem{proposition}[theorem]{Proposition}
\newtheorem{lemma}[theorem]{Lemma}
\newtheorem{corollary}[theorem]{Corollary}
\newtheorem{definition}[theorem]{Definition}

\newtheorem{example}[theorem]{Example}
\newtheorem{algorithm}[theorem]{Algorithm}

\newenvironment{proof}{\noindent Proof:}{$\Box$}

\newcommand{\Z}{{\mathbb Z}}
\newcommand{\Q}{{\mathbb Q}}
\newcommand{\N}{{\mathbb N}}

\newcommand{\Ann}{{\mathrm {Ann}}}

\newcommand{\spoly}{{\mathrm {sp}}}
\newcommand{\init}{{\mathrm {in}}}
\newcommand{\mult}{{\mathrm {mult}\,}}
\newcommand{\length}{{\mathrm {length}\,}}

\newcommand{\Lsc}{{\mathcal L}}
\newcommand{\Ksc}{{\mathcal K}}

\newcommand{\Char}{{\mathrm{Char}}}

\newcommand{\dx}{{\partial_x}}
\newcommand{\dy}{{\partial_y}}
\newcommand{\dz}{{\partial_z}}

\title{Localization, local cohomology, 
and the $b$-function\\ of a $D$-module with respect to a polynomial}

\author{Toshinori Oaku}
\date{}

\begin{document}
\maketitle

\begin{abstract}
Given a  $D$-module $M$ generated by a single element, and a polynomial $f$, 
one can construct several $D$-modules attached to $M$ and $f$ 
and can define the notion of the (generalized) $b$-function following M.~Kashiwara. 
These modules are closely related to the localization 
and the local cohomology of $M$. 
We show that the $b$-function, if it exists,  controls these modules and 
present general algorithms for computing these modules and the $b$-function 
(if it exists) without any further assumptions. 
We also give some examples of multiplicity computation of such $D$-modules 
\end{abstract}

\section*{Introduction}

Let $K$ be a field of characteristic zero 
and $K[x] = K[x_1,\dots,x_n]$ be the polynomial ring 
in $n$ variables $x = (x_1,\dots,x_n)$. 
Let $D_n = K[x]\langle \partial\rangle 
= K[x]\langle \partial_{1},\dots, \partial_{n} \rangle $ 
be the $n$-th Weyl algebra, i.e., the ring of differential 
operators with polynomial coefficients with respect to 
the variables $x$, where we denote 
$\partial = (\partial_{1},\dots,\partial_{x})$ 
with $\partial_{i} = \partial_{x_i}= \partial/\partial x_i$ being the
derivation with respect to $x_i$. 
An arbitrary element $P$ of $D_n$ is written in a finite sum
\[
P = \sum_{\alpha\in\N^n} a_\alpha(x)\partial^\alpha 
\quad \mbox{with}\quad
a_\alpha(x) \in K[x], 
\]
where we denote 
$\partial^\alpha 
= \partial_{1}^{\alpha_1}\cdots\partial_{n}^{\alpha_n}$ 
for a multi-index $\alpha = (\alpha_1,\dots,\alpha_n) \in \N^n$ 
with $\N$ being the set of non-negative integers.  
One can define the dimension of a finitely generated left $D_X$-module $M$; 
J. Bernstein \cite{Bernstein1}, \cite{Bernstein2} proved that the dimension of $M$ is 
not less than $n$ unless $M$ is the zero module. 
A finitely generated left $D_n$-module is 
called {\em holonomic} if its dimension is $n$ or else it is the zero module. 

Let $M$ be a finitely generated left $D_n$-module and $f \in K[x]$ be a non-constant polynomial. 
Then the localization $M[f^{-1}]$ and the local cohomology groups 
$H^j_{(f)}(M)$ have natural structures of left $D_n$-module and are holonomic 
if so is $M$, 
as was shown by Kashiwara \cite{KashiwaraII}. 
More generally,  one can construct a left $D_n[s]$-module
$M(u,f,s) = D_n[s](u\otimes f^s)$ with an indeterminate $s$. 
Suppose that $M$ is generated by $u$ over $D_n$. 
Then the {\em (generalized) $b$-function} for $u$ and $f$ is defined to be 
the univariate (and monic) polynomial $b_{u,f}(s)$ of the least degree such that
\[
 b_{u,f}(s)(u\otimes f^{s}) \in D_n[s](u\otimes f^{s+1})
\]
holds. The existence of $b_{u,f}(s)$ was proved by Kashiwara \cite{KashiwaraII} 
under the assumption that $M$ is holonomic outside of the hypersurface $f=0$. 
If $M$ is the polynomial ring $K[x]$ with $u=1$, then $b_{u,f}(s)$ is nothing but 
the classical Bernstein-Sato polynomial, or simply the $b$-function, of $f$. 
In the same way as the Bernstein-Sato polynomial controls the localization of the 
polynomial ring as a $D_n$-module, the $b$-function controls the localization 
$M[f^{-1}]$ or its generalization $D_n(u\otimes f^\lambda)$. 

On the other hand, algorithms to compute $M(u,f,s)$ and the $b$-function 
if it exists were introduced in \cite{OakuAdvance} under the assumption that $M$ is 
$f$-torsion free. These algorithms are based on various Gr\"obner bases 
over the ring of differential operators as is presented, e.g., in 
\cite{SST} and \cite{OakuLecture}. 
Torrelli \cite{Torrelli} studied the $b$-function $b_{u,f}(s)$ systematically 
when $M$ is the local cohomology group $H^k_{(f_1,\dots,f_k)}(K[x])$ 
under the assumption that $f_1,\dots,f_k,f$ define a quasi-homogeneous 
non-isolated singularity, together with the general property of $M(u,f,s)$ 
under the assumption that $M$ is holonomic without $f$-torsion. 

The purpose of our study on the $b$-function and $M(u,f,s)$ is 
twofold: first, we want to clarify how the $b$-function controls 
the module $M(u,f,s)$ and the localization $M[f^{-1}]$ as well as the local 
cohomology $H^1_{(f)}(M)$.  
This will be performed in Sections 1 and 4. 
These results should be more or less well-known under some stronger conditions. 
See, e.g., \cite{Torrelli} and Chapter VI of \cite{Bjork2}, where $M$ is assumed 
to be $f$-torsion free, or regular holonomic. 
The second purpose is to remove the assumption of $f$-saturatedness from our former 
algorithms in \cite{OakuAdvance}. For this purpose, we reinterpret the algorithm 
introduced in \cite{OTW} for the localization $M[f^{-1}]$ in Section 2. 
Our algorithms work at least if $M$ is holonomic outside of $f=0$ without any 
further assumptions. 

In the last section, we study the multiplicity 
(in the sense of Bernstein \cite{Bernstein1}) and the length of 
a holonomic $D$-module, as the most fundamental numerical invariants. 
This can be also used to prove a relation between $b_{u,f}(s)$ and $M(u,f,\lambda)$. 
We also give some examples of the multiplicity computation of the localization or the 
local cohomology.

We use computer algebra system Risa/Asir \cite{asir} for computation of 
Gr\"obner bases over the ring of differential operators, and in particular, 
for computation of $D$-module theoretic integration, which is needed 
in the localization algorithm. 

This work was supported by JSPS Grant-in-Aid for Scientific Research (C) 26400123.

\section{The $b$-function for a $D$-module and a polynomial after Kashiwara} 

Let $K$ be a field of characteristic zero and 
$X = K^n$ be the $n$-dimensional affine space over $K$. 
We denote by $D_X$ the $n$-th Weyl algebra $D_n$ over $K$. 
Let $M$ be a left $D_X$-module and $f \in K[x]$ a non-constant polynomial. 
We can associate several $D_X$-modules with $M$ and $f$ 
by translating the definitions by  Kashiwara \cite{KashiwaraII} for 
analytic $D$-modules to algebraic setting.  
First, the localization $M[f^{-1}] := M \otimes_{K[x]}K[x,f^{-1}]$ 
and the local cohomology groups 
$H^j_{(f)}(M)$ ($j=0,1$) are defined with $M$ being regarded as 
a $K[x]$-module; they become again left $D_X$-modules. 

Introducing an indeterminate $s$, let 
\[
\Lsc := K[x,f^{-1},s]f^s
\]
be the free $K[x,f^{-1},s]$-module with a free generator $f^s$. 
Then $\Lsc$ has a natural structure of left $D_X[s]$-module through 
the action of $x_i$ on $\Lsc$ defined by
\[
\partial_{x_i}(a(x,s)f^{-k}f^s) = 
\left(\frac{\partial a(x,s)}{\partial x_i}f^{-k} 
+ (s-k)f_ia(x,s)f^{-k-1}\right)
f^s \qquad (1 \leq i \leq n)
\]
with $f_i := \partial f/\partial x_i$. 
Sometimes $f^{-k}f^s$ is abbreviated to $f^{s-k}$. 

The tensor product
$M \otimes_{K[x]}\Lsc$ 
has a natural structure of left $D_X[s]$-module induced by
\[
\partial_{x_i}(u\otimes a(x,s)f^s) 
= (\partial_{x_i}u)\otimes a(x,s)f^s + u \otimes \partial_{x_i}(a(x,s)f^s)
\quad (1 \leq i \leq n)
\]
for $u\in M$ and $a(x,s) \in K[x,s]$. 
In what follows, we fix an arbitrary nonzero element $u$ of $M$. 
Let $M(u,f,s) := D_X[s](u\otimes f^s)$  be the left $D_X[s]$-submodule 
of $M\otimes_{K[x]}\Lsc$ generated by $u\otimes f^s$. 
Set 
\[
I(u,f) := \{b(s) \in K[x] \mid b(s)(u\otimes f^s) = P(s)(fu\otimes f^{s}) = 0 
\mbox{ for some $P(s) \in D_X$} \} .
\]
If $I(u,f) \neq \{0\}$, then the (monic) generator 
$b_{u,f}(s)$ of $I(u,f)$ 
is called the {\em (generalized) $b$-function} for $u$ and $f$. 
It was defined by Kashiwara \cite{KashiwaraII} with 
the following existence theorem. 

\begin{theorem}[Kashiwara \cite{KashiwaraII}]
If $M$ is holonomic on $X_f = \{x \in X \mid f(x) \neq 0\}$, 
then $I(u,f) \neq \{0\}$. 
\end{theorem}
When $M = K[x]$ and $u = 1$, the $b$-function 
$b_{1,f}(s)$ is nothing but what is called the 
Bernstein-Sato polynomial, or the $b$-function, associated with $f$. 
In fact, Kashiwara proved this theorem for a module $M$ over 
the ring of differential operators with analytic coefficients and
a complex analytic function $f$. This corresponds to what is 
called the local $b$-function. 
The coincidence of the local $b$-functions in the algebraic 
setting and in the analytic setting is noticed, e.g., 
as Corollary 8.6 of \cite{OakuAdvance}.  
It will turn out in what follows that the $b$-function 
`controls' the $D$-modules associated with $M$ and $f$. 

The $b$-function can exist even if $M$ is not holonomic on $X_f$.

\begin{example}\label{ex:P}\rm
Set $n=2$, $x_1 = x$, $x_2=y$, and $P = x\partial_x^2+\partial_y$.  
Then $M := D_X/D_XP = D_Xu$ with $u$ being the residue class of $1$ 
is not holonomic even outside of $x=0$ (the dimension of $M$ is three), 
but has the $b$-functions 
$b_{u,x}(s) = (s+1)(s+2)$ and $b_{u,y}(s) = s+1$.  In fact, one has
\[
(-x\partial_x^2 + 2(s+1)\partial_x - \partial_y)(u\otimes x^{s+1}) 
= (s+1)(s+2)u\otimes x^s, 
\qquad
P(u\otimes y^{s+1}) = (s+1)u\otimes y^s
\]
in $M\otimes_{K[x,y]}K[x,y,x^{-1}]x^s$ and 
in $M\otimes_{K[x,y]}K[x,y,y^{-1}]y^s$ respectively. 
\end{example}

An algorithm to determine if there exists the $b$-function and to 
compute it if it exists was given 
in \cite{OakuAdvance} under the assumption that $M = D_Xu$ 
is $f$-torsion free, or $f$-saturated, i.e., 
the homomorphism $f : M \rightarrow M$ is injective. 

Let us define a $D_X$-automorphism $t : \Lsc \rightarrow \Lsc$
by 
\[
t(a(x,s)f^{-k}f^s) = a(x,s+1)f^{-k+1}f^s 
\]
for $a(x,s) \in K[x,s]$ and $k \in \N$.
The inverse $t^{-1}$ is defined by
\[
 t^{-1}(a(x,s)f^{-k}f^s) = a(x,s-1)f^{-k-1}f^s.  
\]
It induces a $D_X$-automorphism
\[
t : M\otimes_{K[x]}\Lsc \longrightarrow M\otimes_{K[x]}\Lsc, 
\] 
which also induces a $D_X$-endomorphism of $M(u,f,s)$.　
Note that the actions of $t$ and $s$ on $M(u,f,s)$ satisfies the 
commutation relation $st = t(s-1)$. 
It follows that $tM(u,f,s)$ is a left $D_X[s]$-module. 
It also follows from the definition that $b_{u,f}(s)$ is the minimal 
polynomial of the action $s$ on the left $D_X$-module 
$M(u,f,s)/tM(u,f,s)$
since $P(s)(fu\otimes f^s) = t(P(s-1)(u\otimes f^s))$. 

Let $\lambda \in K$ be a constant. 
Then specializing the parameter $s$ to $\lambda$, we obtain 
left $D_X$-modules
\[
\Lsc(\lambda) := \Lsc/(s-\lambda)\Lsc,
\qquad
M(u,f,\lambda) := M(u,f,s)/(s-\lambda)M(u,f,s).
\]
Let us denote by $f^s|_{s=\lambda}$ and 
$(u\otimes f^s)|_{s=\lambda}$ the residue class of $f^s$ in $\Lsc(\lambda)$, 
and that of $u\otimes f^s$ in $M(u,f,\lambda)$ respectively. 
Note that $M(u,f,\lambda)$ is not, at least a priori,  a submodule of 
$M\otimes_{K[x]}\Lsc(\lambda)$. 
Kashiwara also proved the following fundamental fact, 
to which we shall give an elementary proof in Section 4.  

\begin{theorem}[Kashiwara \cite{KashiwaraII}]
If $M$ is holonomic on $X_f$, 
then $M(u,f,\lambda)$ is a holonomic $D_X$-module for any 
$\lambda \in K$. 
\end{theorem}

On the other hand, 
the free $K[x,f^{-1}]$-module $K[x,f^{-1}]f^\lambda$ 
with a free generator $f^\lambda$  has a natural structure of 
left $D_X$-module induced by
\[
\partial_{x_i}(a(x)f^{-k}f^\lambda) 
= \left(\frac{\partial a(x)}{\partial x_i}f^{-k} 
+ (\lambda -k) f_ia(x)f^{-k-1}\right)f^\lambda
\qquad (1 \leq i \leq n)
\]
for $a(x) \in K[x]$ and $k \in \N$. 
In particular, this module is isomorphic to the localization 
$M[f^{-1}]$ if $\lambda$ is an integer. 
We remark that $K[x,f^{-1}]f^\lambda$ is not isomorphic to $\Lsc(\lambda)$ 
in general. For example, set $n=1$, $f =x = x_1$, and $\lambda=0$. 
Then $x^s|_{s=0}$ satisfies $x\partial_x (x^s|_{s=0}) = 0$ but 
$\partial_x (x^s|_{s=0}) \neq 0$, 
while $\partial_xx^0 = 0$ by the definition.

Let us define the specialization homomorphism
\[
\rho_\lambda : M\otimes_{K[x]}\Lsc  \longrightarrow 
M\otimes_{K[x]} K[x,f^{-1}]f^\lambda
\]
by 
\[
\rho_\lambda(v\otimes a(x,s)f^{-k}f^s)  
= v \otimes  a(x,\lambda)f^{-k}f^\lambda
\]
for $v\in M$, $a(x,s) \in K[x,s]$, and $k\in\N$. 
Then 
$
\rho_\lambda(P(s)w) = P(\lambda)\rho_\lambda(w)
$
holds for any $w \in M\otimes_{K[x]}\Lsc$ and $P(s) \in D_X[s]$. 
Since any element of $(s-\lambda)M(u,f,s)$ is sent by $\rho_\lambda$ 
to zero, 
$\rho_\lambda$ induces a surjective $D_X$-homomorphism 
\[
\tilde\rho_\lambda : M(u,f,\lambda) \longrightarrow D_X(u\otimes f^\lambda)
\subset M\otimes_{K[x]}K[x,f^{-1}]f^\lambda. 
\]

\begin{lemma}\label{lemma:expression}
Let $M = D_Xu$ be a left $D_X$-module generated by $u$. 
\begin{enumerate}
\item
Every element of $M\otimes_{K[x]} \Lsc$ can be expressed as 
$Q(s)(u\otimes f^{s-k})$ with some $Q(s) \in D_X[s]$ and $k \in \N$. 
\item
Let $\lambda$ be an arbitrary element of $K$. Then 
every element of $M\otimes_{K[x]}K[x,f^{-1}]f^\lambda$  
can be expressed as 
$Q(u\otimes f^{\lambda-k})$ with some $Q \in D_X$ and $k \in \N$. 
\end{enumerate}
\end{lemma}

\begin{proof}
From the identity
$\partial_{x_i}(v\otimes f^{s-k}) = (\partial_{x_i}v)\otimes f^{s-k} 
+ v\otimes (s-k)f_if^{s-k-1}$  for any $v\in M$ and $k\in\Z$, we get
\[
(\partial_{x_i}v)\otimes f^{s-k} 
= (\partial_{x_i}f - (s-k)f_i)(v\otimes f^{s-k-1}). 
\]
By induction, we can show that for any multi-index $\alpha \in \N^n$ 
and $k\in\Z$, there exists $Q_\alpha(s) \in D_X[s]$ such that
\[
(\partial_x^\alpha v)\otimes f^{s-k} 
= Q_\alpha(s)(v \otimes f^{s-k-|\alpha|}). 
\]
This proves the statement (1). The statement (2) can be proved similarly. 
\end{proof}

The following proposition should be well-known; 
see, e.g., Propositions 7.1 and 7.4 of \cite{OakuAdvance}. 
The case $M = K[x]$ and $f=1$ was first proved by 
Kashiwara \cite{Kashiwara1}. 

\begin{proposition}\label{prop:fs1}
Let $M$ be a left $D_X$-module generated by $u \in M$ 
and assume that there exists the $b$-function 
$b_{u,f}(s)$. Let $\lambda$ be an element of $K$ and 
suppose that $b_{u,f}(\lambda-k) \neq 0$ for any positive 
integer $k$. Then 
\begin{enumerate}
\item
The image $\rho_\lambda(M(u,f,s)) = D_X(u\otimes f^\lambda)$ coincides with 
$M\otimes_{K[x]}K[x,f^{-1}]f^\lambda$. 
In other words,  
$M\otimes_{K[x]}K[x,f^{-1}]f^\lambda$ is generated by 
$u\otimes f^\lambda$ over $D_X$. 
\item
$\ker\rho_\lambda \cap M(u,f,s)$ coincides with $(s-\lambda)M(u,f,s)$. 
Hence 
$\tilde\rho_\lambda : M(u,f,\lambda) \rightarrow D_X(u\otimes f^\lambda)$ 
is an isomorphism of left $D_X$-modules. 
\end{enumerate}
\end{proposition}

\begin{proof}
(1) 
In view of Lemma \ref{lemma:expression}, 
we have only to show that $u\otimes f^{\lambda-k}$ 
belongs to $\rho_\lambda(M(u,f,s))$ for any $k \in \N$. 
This is obvious for $k=0$ since 
$\rho_\lambda (u\otimes f^s) = u\otimes f^\lambda$.

Let us show that $u\otimes f^{\lambda-k}$  
belongs to $\rho_\lambda(M(u,f,s))$.
Suppose $k \geq 1$.  
There exists $P(s) \in D_X[s]$ such that 
$P(s)(u\otimes f^{s+1}) = b_{u,f}(s)(u\otimes f^s)$. 
Applying $t^{-k}$, we get 
\[
P(s-k)(u\otimes f^{s+1-k}) = b_{u,f}(s-k)(u\otimes f^{s-k})
\]
in $M\otimes_{K[x]}\Lsc$. 
Proceeding inductively, we see that there exists 
$\tilde P(s) \in D_X[s]$ such that
\begin{equation}\label{eq:shift}
\tilde P(s)(u\otimes f^s) = b_{u,f}(s-1)\cdots b_{u,f}(s-k) u\otimes f^{s-k} 
\end{equation}
holds in $M\otimes_{K[x]}\Lsc$.
The homomorphism $\rho_\lambda$ gives an identity
\[
\tilde P(\lambda)(u\otimes f^\lambda) 
= b_{u,f}(\lambda-1)\cdots b_{u,f}(\lambda-k) u\otimes f^{\lambda-k} 
\]
in $M\otimes_{K[x]}K[x,f^{-1}]f^\lambda$. 
Since $b_{u,f}(\lambda-j) \neq 0$ for $j=1,\dots,k$ by the assumption, 
it follows that
\[
u\otimes f^{\lambda-k} = 
\frac{1}{b_{u,f}(\lambda-1)\cdots b_{u,f}(\lambda-k)} 
\tilde P(\lambda)(u\otimes f^{\lambda}).
\]
The right-hand side belongs to $\rho_\lambda(M(u,f,s))$. 
This completes the proof of (1).

(2) Assume $\rho_\lambda(Q(s)(u\otimes f^s)) = 0$ with $Q(s) \in D_X[s]$. 
There exist $l \in\N$ and $Q_j \in D_X$ which are zero except finitely many 
indices $j$ such that
\[
 Q(s)(u\otimes f^s) = \sum_{j \geq 0} (Q_ju) \otimes (s-\lambda)^j f^{s-l}. 
\]
By the assumption, 
$\rho_\lambda(Q(s)(u\otimes f^s)) = (Q_0u)\otimes f^{\lambda-l}$ 
vanishes in $M\otimes_{K[x]}K[x,f^{-1}]f^\lambda$, which means that 
$(Q_0u)\otimes f^{-l}$ vanishes in $M\otimes_{K[x]}K[x,f^{-1}]$. 
It follows that $(Q_0u)\otimes 1 = f^l(Q_0u)\otimes f^{-l} = 0$ in 
$M\otimes_{K[x]}K[x,f^{-1}]$. 
Consequently, $(Q_0u)\otimes f^\lambda$ vanishes in $M\otimes_{K[x]}\Lsc$. 
Thus we have
\[
Q(s)(u\otimes f^s) = 
(s-\lambda)\sum_{j \geq 1} (Q_ju) \otimes (s-\lambda)^{j-1} f^{s-l}
= (s-\lambda)Q'(s)(u\otimes f^{s-k})
\]
with some $k\in\N$ and $Q'(s) \in D_X[s]$ in view of the proof 
of Lemma \ref{lemma:expression}. 
By using (\ref{eq:shift}) we obtain
\[
b_{u,f}(s-1)\cdots b_{u,f}(s-k)Q(s)(u \otimes f^s) 
= (s-\lambda) Q'(s)\tilde P(s)(u\otimes f^{s}). 
\]
Hence 
$b_{u,f}(\lambda-1)\cdots b_{u,f}(\lambda-k)Q(s)(u \otimes f^s)$
belongs to $(s-\lambda)M(u,f,s)$. 
This completes the proof of (2). 
\end{proof}

The following proposition extends Lemma 1.3 of Walther \cite{Walther2} 
for the case $M=K[x]$ and $u=1$ almost verbatim. 

\begin{lemma}\label{lemma:fs2}
Under the same assumption as in the preceding proposition, assume moreover 
that $b_{u,f}(\lambda) = 0$. 
Then one has 
\[
 D_X(fu\otimes f^{\lambda}) \subsetneqq D_X(u\otimes f^\lambda)
\]
in $M\otimes_{K[x]}K[x,f^{-1}]f^\lambda$. 
In particular, $M \otimes_{K[x]}K[x,f^{-1}]f^\lambda$ 
is generated by $u\otimes f^\lambda$, but not by 
$u\otimes f^{\lambda + 1} = fu\otimes f^\lambda$, over $D_X$. 
\end{lemma}

\begin{proof}
There exists $P(s) \in D_X[s]$ such that 
$P(s)(fu\otimes f^s) = b_{u,f}(s)(u\otimes f^s)$. 
Assume $D_X(fu\otimes f^{\lambda}) = D_X(u\otimes f^\lambda)$. 
Then there exists $A \in D_X$ such that
$(1-Af)(u\otimes f^\lambda) = 0$. 
By virtue of (2) of the preceding proposition, there exist $Q(s),R(s)
 \in D_X[s]$ such that
\[
 1-Af = Q(s) + (s-\lambda)R(s), \qquad Q(s)(u\otimes f^s) = 0.
\]
It follows that
\begin{align*}
\frac{b_{u,f}(s)}{s-\lambda} (u\otimes f^s) 
&= \frac{b_{u,f}(s)}{s-\lambda} Af(u\otimes f^s)
+ b_{u,f}(s)R(s)(u\otimes f^s)
\\&= 
\left(\frac{b_{u,f}(s)}{s-\lambda}A + R(s)P(s)\right)(fu\otimes f^s)
.
\end{align*}
This means that $b_{u,f}(s)/(s-\lambda)$ belongs to the ideal 
$I(u,f)$, which contradicts the definition of $b_{u,f}(s)$. 
This completes the proof. 
\end{proof}

Summing up we obtain

\begin{theorem}\label{th:gen}
Let $M = D_Xu$ be a left $D_X$-module generated by $u \in M$ and $f \in K[x]$ 
be a non-constant polynomial. Assume that there exists the $b$-function 
$b_{u,f}(s)$ for $u$ and $f$. Then the following conditions 
are equivalent:
\begin{enumerate}
\item
$b_{u,f}(\lambda - k) \neq 0$ for any positive integer $k$. 
\item
$M\otimes_{K[x]}K[x,f^{-1}]f^\lambda$ is generated by $u\otimes f^\lambda$ 
over $D_X$. 
\end{enumerate}
\end{theorem}

\begin{proof}
Assume $b_{u,f}(\lambda-k) = 0$ for some positive integer $k$ and 
let $k_0$ be the maximum among such $k$. 
Then by (1) of Proposition \ref{prop:fs1} and Lemma \ref{lemma:fs2}, we have 
\[
 K[x,f^{-1}]f^\lambda = D_X(u\otimes f^{\lambda-{k_0}}) \supsetneqq 
D_X(fu\otimes f^{\lambda-{k_0} +1}) \supset 
D_X(fu\otimes f^{\lambda}). 
\] 
Hence $K[x,f^{-1}]f^\lambda$ is not generated by $u\otimes f^\lambda$. 
\end{proof}

The converse of the statement (2) of Proposition \ref{prop:fs1} 
will be given in Theorem \ref{th:rho} of Section 4 under the additional assumption 
that $M$ is holonomic on $X_f$.

Let us recall local cohomology of $D$-modules. 
Let $M$ be a finitely generated left $D_X$-module, and 
$I$ be an ideal of $K[x]$. Then the $k$-th local cohomology group 
$H^k_{I}(M)$ supported by $I$ 
is defined to be the $k$-th derived functor of the functor 
\[
 M \longmapsto H^0_I(M) = \{u \in M \mid I^k u = 0 
\mbox{ for some $k \in \N$} \}. 
\]
They have natural structure of left $D_X$-module, and 
they are holonomic if so is $M$ as was proved by Kashiwara 
in the analytic category \cite{KashiwaraII}.  

If $I$ is the principal ideal $(f)$ generated by $f \in K[x]$, then 
there exists an exact sequence
\[
0 \longrightarrow H^0_{(f)}(M) \longrightarrow M 
\stackrel{\iota}{\longrightarrow} M[f^{-1}] \longrightarrow H^1_{(f)}(M)
\rightarrow 0
\]
of left $D_X$-modules, 
where $\iota$ stands for the natural homomorphism such that 
$\iota(v) = v\otimes 1$ in $M[f^{-1}] = M\otimes_{K[x]}K[x,f^{-1}]$ 
for $v \in M$. 
Hence there is an isomorphism $H^1_{(f)}(M) \cong M[f^{-1}]/\iota(M)$ 
as left $D_X$-module. 

In general, algorithms to compute $H^i_I(M)$ as left $D_X$-module 
were given in \cite{OakuAdvance} for the case $I$ is principal, 
and in \cite{Walther1} and \cite{OTalgDmod} for general $I$,  
under the condition that $M$ is holonomic.

\begin{corollary}
Let $M = D_Xu$ be a left $D_X$-module generated by $u \in M$ and $f \in K[x]$ 
be a non-constant polynomial. Assume that there exists the $b$-function 
$b_{u,f}(s)$ for $u$ and $f$. Then the following conditions 
are equivalent:
\begin{enumerate}
\item
$b_{u,f}(j) \neq 0$ for any integer $j <k$. 
\item
The localization $M[f^{-1}]$ is generated by $u\otimes f^{-k}$ over $D_X$. 
\item
The local cohomology group $H_{(f)}^1(M)$ is generated by the cohomology 
class $[u\otimes f^{-k}]$ over $D_X$. 
\end{enumerate}
\end{corollary}

\begin{proof}
The equivalence of (1) and (2) is a special case of Theorem \ref{th:gen}. 
In general, if $M[f^{-1}]$ is generated by $u\otimes f^{-k}$, 
then $H_{(f)}^1(M) = M[f^{-1}]/\iota(M)$ is generated by its residue class. 
Conversely, assume that $M[f^{-1}]/\iota(M)$ is generated by 
$[u\otimes f^{-k}]$. Then for any $w \in M[f^{-1}]$, there exist $P,Q \in D_X$ 
such that 
\[
 w = P(u\otimes f^{-k}) + (Qu)\otimes 1 = (P + Qf^k)(u\otimes f^{-k}).
\] 
Henece $M[f^{-1}]$ is generated by $u\otimes f^{-k}$.
\end{proof}

\section{Localization algorithm revisited}

Let $M$ be a left $D_n$-module and $f \in K[x]$ be a nonzero polynomial. 
Let $X = K^n$ be the $n$-dimensional affine space over $K$. 
Then $X_f := \{x \in X \mid f(x) \neq 0\}$ is an affine open subset of $X$. 
Our purpose is to reformulate the algorithm given in \cite{OTW} for 
computing the localization $M[f^{-1}] := M\otimes_{K[x]}K[x,f^{-1}]$ 
as left $D_X$-module 
by using local cohomology, 
hoping to clarify the meaning of the algorithm  as well as to 
make the canonical 
homomorphism $\iota : M \rightarrow M[f^{-1}]$ more explicit. 

We assume in what follows, as well as in \cite{OTW},  that $M$ is holonomic on $X_f$; i.e., 
$\Char(M) \cap \pi^{-1}(X_f)$ is an $n$-dimensional algebraic set 
of $\pi^{-1}(X_f)$, where $\Char(M)$ is the characteristic variety of $M$, 
which is an algebraic set of the cotangent bundle $T^*X = 
\{(x,\xi) \in K^n \times K^n\}$ and  $\pi : T^*X \rightarrow X$ is the 
projection.  

Introducing a new variable $t$, set $Y = X \times K \ni (x,t)$ and 
\[
Z := \{(x,t) \in Y \mid tf(x) = 1\}.
\]
Then $Z$ is an affine subset of $Y$ which is isomorphic to $X_f$. 
Let 
\[
B_{Z|Y} := H^1_{(tf(x) -1)}(K[x,t]) = K[x,t,(tf-1)^{-1}]/K[x,t]
\]
be the first local cohomology group of $K[x,t]$ with support in $Z$, 
which we regard as a left $D_Y$-module. 
An arbitrary element of $B_{Z|Y}$ is expressed as 
\[
\left[ \frac{a(x,t)}{(tf(x) -1)^{k+1}}\right] 
\qquad (k \in\N,\, a(x,t) \in K[x,t]),
\]
where the bracket denotes the residue class in $B_{Z|Y}$. 


Set $f_i = \partial f/\partial x_i$ for $i=1,\dots,n$ and define
\[
\delta^{(k,l)} := \left[ \frac{f^{l+1}}{(tf-1)^{k+1}} \right] 
\]
for $k,l\in \Z$ with $l \geq -1$. 
Note that $\delta^{(k,l)} = 0$ by the definition if $k < 0$. 
As left $K[x,t]$-module, $B_{Z|Y}$ is generated by 
$\delta^{(k,-1)}$ with $k \in \N$. 

We have the following identities for $k,l \geq 0$: 
\begin{align*}
\partial_t\delta^{(k,l)} 
& = 
-(k+1)\left[\frac{f^{l+2}}{(tf-1)^{k+2}}\right]
= -(k+1)\delta^{(k+1,l+1)},
\\
\partial_{x_i}\delta^{(k,l)}
& = (l+1)\left[\frac{f_if^l}{(tf-1)^{k+1}}\right]
 -(k+1)\left[\frac{tf_if^{l+1}}{(tf-1)^{k+2}}\right]
\\
&= (l+1)f_i\delta^{(k,l-1)} 
  -(k+1)\left[\frac{f_i(tf-1+1)f^{l}}{(tf-1)^{k+2}}\right]
\\&=
(l+1)f_i\delta^{(k,l-1)} 
  -(k+1)f_i(\delta^{(k,l-1)} + \delta^{(k+1,l-1)})
\\&
= (l-k)f_i\delta^{(k,l-1)} -(k+1)f_i \delta^{(k+1,l-1)}
,
\\
t\delta^{(k,l)} &=
\left[\frac{tf^{l+1}}{(tf-1)^{k+1}}\right] 
= \left[\frac{\{(tf-1)+1\}f^l}{(tf-1)^{k+1}}\right] 
= \delta^{(k-1,l-1)} + \delta^{(k,l-1)}. 
\end{align*}
In particular, we have
\[
(\partial_t t + k)\delta^{(k,l)} = -(k+1)\delta^{(k+1,l)},
\qquad
t\delta^{(0,0)} = \delta^{(0,-1)}. 
\]
Hence $B_{Z|Y}$ is generated by $\delta^{(0,0)} = [f(tf-1)^{-1}]$ 
as a left $D_Y$-module.

\begin{lemma}
One has 
$(tf-1)\delta^{(0,0)} = 0$ and 
$(\partial_{x_i} -f_i\partial_tt^2)\delta^{(0,0)} = 0$
for $i=1,\dots,n$. 
\end{lemma}
 
\begin{proof}
The first equality follows immediately from the definition. 
The second equality follows from
\[
\partial_tt^2\delta^{(0,0)} = (t^2\partial_t + 2t)\delta^{(0,0)}
= -\delta^{(1,-1)}
\]
by using the formulae above. 
\end{proof}

Let us regard $B_{Z|Y}$ as a module over the subring $K[x]$ of $D_Y$ 
and consider the localization
\[
 B_{Z|Y}[f^{-1}] := B_{Z|Y}\otimes_{K[x]}K[x,f^{-1}] 
= K[x,t,f^{-1},(tf-1)^{-1}]/K[x,t,f^{-1}]
\]
with respect to $f$. Let us denote the residue class in $B_{Z|Y}[f^{-1}]$ 
by $[\bullet]'$ in order to distinguish it from the residue class in 
$B_{Z|Y}$ which is denoted $[\bullet]$. 

\begin{lemma}
The natural homomorphism
\[
\iota' : B_{Z|Y} \ni \left[ \frac{a(x,t)}{(tf-1)^{k+1}} \right]
\longmapsto 
\left[ \frac{a(x,t)}{(tf-1)^{k+1}} \right]'
\in B_{Z|Y}[f^{-1}] 
\]
is an isomorphism of left $D_Y$-modules. 
\end{lemma}

\begin{proof}
Assume $\iota'([a(x,t)(tf-1)^{-k-1}]) = 0$ with $a(x,t) \in K[x,t]$.   
Then there exists an integer $l$ such that 
$f^la(x,t)$ is divisible by $(tf-1)^{k+1}$ in $K[x,t]$. 
Since $f$ and $tf-1$ are relatively prime, 
$a(x,t)$ must be divisible by $(1-tf)^{k+1}$. This proves that $\iota$ 
is injective. 

Let us show that $\iota'$ is surjective. 
It suffices to show that $[f^{-m}(tf-1)^{-k-1}]'$  
belongs to the image of $\iota'$ for any $k,m \in \N$ 
by induction on $k+m$, which obviously holds for $k=m=0$. 
Suppose $k+m \geq 1$.  
We have
\begin{align*}&
\left[ \frac{tf}{(tf-1)^{k+1}} \right]
= \left[ \frac{1 + (tf-1)}{(tf-1)^{k+1}} \right]
=
\left[ \frac{1}{(tf-1)^{k+1}} \right]
+
\left[ \frac{1}{(tf-1)^{k}}\right]
\end{align*}
It follows that
\begin{align*}
\left[ \frac{f^{-m}}{(tf-1)^{k+1}} \right]'
&=
\left[ \frac{tf^{1-m}}{(tf-1)^{k+1}} \right]'
- 
\left[ \frac{f^{-m}}{(tf-1)^{k}}\right]'.
\end{align*}
By the induction hypothesis, the right-hand side belongs to the image 
of $\iota'$. 
This completes the proof. 
\end{proof}

\begin{proposition}\label{prop:iso1}\rm
Let $M$ be a finitely generated left $D_n$-module. 
Then the homomorphism
\[
B_{Z|Y}\otimes_{K[x]}M  \stackrel{\sim}{\longrightarrow}   
B_{Z|Y}\otimes_{K[x]}M[f^{-1}]
\]
of left $D_Y$-modules, 
which is induced by the natural homomorphism $\iota : M \rightarrow M[f^{-1}]$ 
is an isomorphism. 
\end{proposition}

\begin{proof}
We have 
\begin{align*}
B_{Z|Y}\otimes_{K[x]}M[f^{-1}] 
&= B_{Z|Y}\otimes_{K[x]}(K[x,f^{-1}]\otimes_{K[x]}M)
\\&
= (B_{Z|Y}\otimes_{K[x]}K[x,f^{-1}])\otimes_{K[x]}M
= B_{Z|Y}[f^{-1}]\otimes_{K[x]}M .
\end{align*}
Hence the isomorphism $\iota'$ induces an isomorphism
\[
B_{Z|Y}\otimes_{K[x]}M \stackrel{\sim}{\longrightarrow}
B_{Z|Y}[f^{-1}]\otimes_{K[x]}M 
= B_{Z|Y} \otimes_{K[x]}M[f^{-1}].
\]
\end{proof}

\begin{proposition}\label{prop:iso2}\rm
Let $M$ be a finitely generated left $D_n$-module. 
Then there exists an isomorphism
\[
B_{Z|Y}\otimes_{K[x]}M \stackrel{\sim}{\longrightarrow}  
  B_{Z|Y}[f^{-1}]\otimes_{K[x,f^{-1}]}M[f^{-1}]
\]
of left $D_Y$-modules. 
\end{proposition}

\begin{proof}
We have
\[
B_{Z|Y}[f^{-1}]\otimes_{K[x,f^{-1}]}M[f^{-1}]
= (B_{Z|Y}\otimes_{K[x]}K[x,f^{-1}])\otimes_{K[x,f^{-1}]}M[f^{-1}]
= B_{Z|Y}\otimes_{K[x]}M[f^{-1}]. 
\]
This proves the assertion combined with the preceding proposition. 
\end{proof}

Let $D_X[f^{-1}] : = K[x,f^{-1}]\otimes_{K[x]}D_X$ and 
$D_Y[f^{-1}] := K[x,f^{-1}]\otimes_{K[x]}D_Y$ 
be the localization of $D_X$ and $D_Y$ by $f$, which can be 
regarded as ring extension of $D_X$ and of $D_Y$ respectively. 
Then $B_{Z|Y}[f^{-1}]$ and 
$B_{Z|Y}[f^{-1}]\otimes_{K[x,f^{-1}]}M[f^{-1}]$ have natural 
structures of left $D_{Y}[f^{-1}]$-module. 

\begin{definition}\rm
We set $\delta^{(j)} = \iota'(\delta^{(j,j)})$ for $j \in \N$. 
We denote $\delta = \delta^{(0)}$. 
\end{definition}

\begin{lemma}\label{lemma:maximal}
As an element of the left $D_Y[f^{-1}]$-module $B_{Z|Y}[f^{-1}]$, 
the annihilator of $\delta$  
coincides with the left ideal of $D_Y[f^{-1}]$ generated by 
\[
t-f^{-1},\quad \partial_{x_i} - f_if^{-2}\partial_t \quad (i=1,\dots,n).
\]
\end{lemma}

\begin{proof}
Let us first verify that these operators annihilate $\delta$. 
In fact, we have
\begin{align*}&
 (t - f^{-1})\delta 
= f^{-1}\iota'((tf-1)\delta^{(0,0)}) = 0,
\\&
(\partial_{x_i} - f_if^{-2}\partial_t)\delta
= \iota'(\partial_{x_i}\delta^{(0,0)})
 - f^{-2}\iota'(f_i\partial_t\delta^{(0,0)})
\\&
= \iota'(-f_i\delta^{(1,-1)}) + f^{-2}\iota'(f_if^2\delta^{(1,-1)})
= 0.
\end{align*}

Assume $P \in D_Y[f^{-1}]$ annihilates $\delta$. 
There exist elements $Q_0,Q_1,\dots,Q_n, R$ of $D_Y[f^{-1}]$ such that
\[
P = Q_0(t - f^{-1}) + \sum_{i=1}^nQ_i(\partial_{x_i} - f_if^{-2}\partial_t) 
+ R
\]
and that $R$ belongs to $K[x,f^{-1},\partial_t]$. 
Writing $R$ in a finite sum  $\sum_{j=0}^l r_j(x)f^{-k}\partial_t^j$ 
with $r_j(x) \in K[x]$ and $k,l \in \N$, we have
\[
0 = R\delta
= \sum_{j=0}^l f^{-k}r_j(x)\partial_t^j\delta
= \sum_{j=0}^l (-1)^jj! f^{-k}r_j(x)\delta^{(j)}
= f^{-k}\left[\frac{\sum_{j=0}^l (-1)^jj! (tf-1)^{l-j}r_j(x)}{(tf-1)^{l+1}}\right]'. 
\]
Since $\iota'$ is injective, this implies
that $r_j(x) = 0$ for any $j\geq 0$, that is, $R=0$. 
\end{proof}


\begin{proposition}\label{prop:BZYbasis}
\rm
An element of $ B_{Z|Y}[f^{-1}]\otimes_{K[x,f^{-1}]}M[f^{-1}]$ is 
uniquely expressed as a finite sum
$
\sum_{j\geq 0} \delta^{(j)}\otimes v_j 
$
with $v_j \in M[f^{-1}]$. 
\end{proposition}

\begin{proof}
By Lemma \ref{lemma:maximal}, $B_{Z|Y}[f^{-1}]$ is isomorphic 
to $K[x,f^{-1},\partial_t]$ as left $D_Y[f^{-1}]$-module.  
Hence $\delta^{(j)} = (-1)^j(1/j!)\partial_t^j\delta$ 
($j \in \N$) constitute a free basis 
of $B_{Z|Y}[f^{-1}]$ over $K[x,f^{-1}]$. 
This implies the assertion of the proposition. 
\end{proof}

\begin{definition}\rm
Set $\vartheta_i := \partial_{x_i} - f_i\partial_t t^2$ 
for $i=1,\dots,n$.  
Define a ring homomorphism $\tau$ from $D_X$ to $D_Y$ by
\[
\tau : D_X \ni P(x,\partial_x) \longmapsto 
P(x,\vartheta_1,\dots,\vartheta_n)
\in D_Y. 
\]
Since $\vartheta_1,\dots,\vartheta_n$ commute with each other, 
and $[\vartheta_i,x_j] = \delta_{ij}$, this substitution is a well-defined
ring homomorphism.  
\end{definition}

\begin{lemma}\label{lemma:tau}
One has
\[
\delta^{(0,0)}\otimes Pv = \tau(P)(\delta^{(0,0)}\otimes v),
\qquad
\delta\otimes Pv' = \tau(P)(\delta\otimes v')
\]
in $B_{Z|Y}\otimes_{K[x]}M$ and in 
$B_{Z|Y}[f^{-1}]\otimes_{K[x,f^{-1}]}M[f^{-1}]$ 
respectively
for any $v \in M$, $v' \in M[f^{-1}]$ and $P \in D_X$.
\end{lemma}

\begin{proof}
We have only to show the first equality. 
Since $ (\partial_{x_i} - f_i\partial_t t^2)\delta^{(0,0)} = 0$, 
we have
\begin{align*}
\tau(\partial_{x_i})(\delta^{(0,0)}\otimes v) 
&= (\partial_{x_i}\delta^{(0,0)})\otimes v 
  + \delta^{(0,0)}\otimes \partial_{x_i}v 
- f_i(\partial_t^2 t^2\delta^{(0,0)})\otimes v 
\\&
= (\partial_{x_i} - f_i\partial_t t^2)\delta^{(0,0)} \otimes v
 + \delta^{(0,0)}\otimes \partial_{x_i}v
= \delta^{(0,0)}\otimes \partial_{x_i}v. 
\end{align*}
We can verify $\tau(P)(\delta^{(0,0)}\otimes v) = \delta^{(0,0)}\otimes(Pv)$ 
by induction on the order of $P$. 
\end{proof}

\begin{proposition}\label{prop:ann}
Let $v \in M[f^{-1}]$, $P \in D_X$, and $k\in\N$. 
Then $P(f^{-k}v) = 0$ holds in $M[f^{-1}]$ 
if and only if $\tau(P)t^k(\delta\otimes v) = 0$ holds 
in $B_{Z|Y}[f^{-1}]\otimes_{K[x,f^{-1}]}M[f^{-1}]$. 
\end{proposition}

\begin{proof}
Since $t^kf^k\delta = (1+t^kf^k-1)\delta = \delta$, we have
\begin{align*}
\delta\otimes P(f^{-k}v) &=
\tau(P)(\delta\otimes f^{-k}v) =
\tau(P)t^kf^k(\delta\otimes f^{-k}v) = \tau(P)t^k(\delta\otimes v). 
\end{align*}
By Lemma \ref{lemma:tau} and Proposition \ref{prop:BZYbasis}, 
this vanishes if and only if $P(f^{-k}v) = 0$. 
\end{proof}

Summing up we obtain

\begin{theorem}\label{th:J}
Let $M = D_Xu$ be a left $D_X$-module generated by $u$ 
and $I = \Ann_{D_X}u$ the annihilator of $u$ so that 
$M = D_X/I$. 
Let $\iota : M \rightarrow M[f^{-1}]$ be the canonical homomorphism 
which sends $u\in M$ to $u\otimes 1$. 
Let $G$ be a finite set of generators of $I$, 
and $J$ be the left ideal of $D_Y$ generated by 
$\{\tau(P) \mid P \in G\}$ and $tf-1$. Then 
\begin{enumerate}
\item
$J$ coincides with 
the annihilator $\Ann_{D_Y} (\delta\otimes \iota(u))$ of 
$\delta\otimes\iota(u)$ in $B_{Z|Y}[f^{-1}]\otimes_{K[x,f^{-1}]}
M[f^{-1}]$. 
\item
$B_{Z|Y}[f^{-1}]\otimes_{K[x,f^{-1}]}M[f^{-1}]$ 
is generated by $\delta\otimes\iota(u)$ as a left $D_Y$-module. 
\item
As a left $D_Y$-module, $B_{Z|Y}\otimes_{K[x]}M$ is isomorphic to $D_Y/J$. 
\end{enumerate}
\end{theorem}

\begin{proof}
(1) 
It is obvious that $J$ is contained in 
$\Ann_{D_Y} (\delta\otimes\iota(u))$. 
Suppose $P(\delta\otimes \iota(u)) = 0$ with $P \in D_Y$. 
There exist $R \in D_Y[f^{-1}]$ and 
$a_{\alpha,j}(x) \in K[x,f^{-1}]$ which are zero  
except finitely many $(\alpha,j) \in \N^n\times\N$ such that
\begin{align*}
P &= \sum_{\alpha \in \N^n, j\geq 0} a_{\alpha,j}(x)\partial_t^j
(\partial_{x_1} - f_1f^{-2}\partial_t)^{\alpha_1}\cdots
(\partial_{x_n} - f_nf^{-2}\partial_t)^{\alpha_n}
+ R(t - f^{-1})
\\&
= \sum_{j\geq 0} \partial_t^j\tau(Q_j) + R(t-f^{-1})
\end{align*}
with $Q_j := \sum_{\alpha \in \N^n} a_{\alpha,j}(x) \partial_x^\alpha 
\in D_X[f^{-1}]$. 
Then we have
\[
0 = P(\delta\otimes \iota(u)) 
= \sum_{j\geq 0}\partial_t^j\tau(Q_j)(\delta\otimes \iota(u)) 
= \sum_{j\geq 0}\delta^{(j)}\otimes Q_j\iota(u)
\]
and consequently $Q_j\iota(u) = 0$ for each $j\geq 0$ by Proposition \ref{prop:BZYbasis}.  
This implies that $f^lQ_ju = 0$ holds in $M$, 
that is, $f^lQ_j$ belongs to $I$,  
for some $l \in\N$ independent of $j$. 
We may also assume that $f^lR$ belongs to $D_Yf$. 
Hence $f^lP = \sum_{j\geq 0}\partial_t^j\tau(f^lQ_j) + f^lR(t-f^{-1})$ belongs to $J$. 
Since $(1- t^lf^l)^kP$ belongs to $D_Y(1-t^lf^l)$, and hence to $J$, 
if we take $k\in\N$ sufficiently large, 
and $t^lf^lP$ belongs to $J$, 
we conclude that $P$ itself belongs to $J$. 

(2) 
By the assumption, Lemma \ref{lemma:expression},  and Proposition \ref{prop:BZYbasis}, 
an arbitrary element of 
$B_{Z|Y}[f^{-1}]\otimes_{K[x,f^{-1}]}M[f^{-1}]$ is expressed as a finite sum
\[
\sum_{j\geq 0} \delta^{(j)}\otimes P_j(u\otimes f^{-k})
\]
with $P_j \in D_X$ and $k \in \N$. We get 
\begin{align*}
\sum_{j\geq 0} \delta^{(j)}\otimes P_j(u\otimes f^{-k}) 
&= \sum_{j\geq 0} \partial_t^j(\delta\otimes P_j(u\otimes f^{-k})) 
= \sum_{j\geq 0} \partial_t^j\tau(P_j)(\delta\otimes(u\otimes f^{-k})) 
\\&=
\sum_{j\geq 0} \partial_t^j\tau(P_j)t^kf^k(\delta\otimes(u\otimes f^{-k}))
= \sum_{j\geq 0} \partial_t^j\tau(P_j)t^k(\delta\otimes\iota(u)). 
\end{align*}
This completes the proof of (2). 

(3) follows from (1), (2) and Proposition  \ref{prop:iso2}. 
\end{proof}

\begin{definition}\rm
For the sake of simplicity of the notation, let us set
\[
 \tilde M := B_{Z|Y}[f^{-1}]\otimes_{K[x,f^{-1}]}M[f^{-1}]
\]
and define a homomorphism $\varphi : M[f^{-1}]\rightarrow 
\tilde M/\partial_t\tilde M$ by 
\[
\varphi(v) = \delta \otimes v \mod \partial_t\tilde M. 
\]
for $v \in M[f^{-1}]$. 
\end{definition}

\begin{theorem}
The homomorphism $\varphi : M[f^{-1}] \rightarrow \tilde M/\partial_t \tilde M$
is an isomorphism of left $D_X[f^{-1}]$-modules, and consequently of $D_X$-modules. 
\end{theorem}

\begin{proof}
By Proposition \ref{prop:BZYbasis}
one has 
\begin{align*}
\tilde M &= (\delta\otimes M[f^{-1}]) \oplus (\delta^{(1)}\otimes M[f^{-1}])
 \oplus (\delta^{(2)}\otimes M[f^{-1}]) \oplus \cdots, 
\\
\partial_t\tilde M &= (\delta^{(1)}\otimes M[f^{-1}]) 
\oplus (\delta^{(2)}\otimes M[f^{-1}]) \oplus \cdots
\end{align*}
as $K[x,f^{-1}]$-modules. 
Hence $\varphi$ is an isomorphism of $K[x,f^{-1}]$-modules. 
For $v \in M[f^{-1}]$ and $P \in D_X[f^{-1}]$, one has
\[
P(\delta\otimes v) \equiv \delta\otimes Pv \mod \partial_t\tilde M.
\]
Hence $\varphi$ is a homomorphism of left $D_X[f^{-1}]$-modules. 
\end{proof}

\begin{theorem}\label{th:Jholonomic}
If $M$ is holonomic on $X_f$, i.e., if  $\Char(M) \cap \pi^{-1}(X_f)$ 
is an $n$-dimensional algebraic set, then $D_Y/J$ is a 
holonomic $D_Y$-module. 
\end{theorem}

\begin{proof}
We may assume $M = D_X/I$. 
By the definition, we have
\begin{align*}
\Char(D_Y/J) &= 
\{(x,t;\xi,\tau) \in K^{2n+2} \mid 
\sigma(P)(x,\xi_1-f_1t^2\tau,\dots,\xi_n-f_nt^2\tau) = 0 
\quad (\forall P \in I),
\\&\quad
tf(x) = 1\}
\\&=
\{(x,t;\xi,\tau)\mid 
(x,\xi_1-f_1t^2\tau,\dots,\xi_n-f_nt^2\tau) \in \Char(M), 
\,\,tf(x) = 1\}.
\end{align*}
Hence there is a bijection
\[
(\Char(M)\cap\pi^{-1}(X_f)) \times K \ni (x,\xi,\tau) 
\longmapsto (x,1/f(x);\xi_1-f_1t^2\tau,\dots,\xi_n-f_nt^2\tau) 
\in \Char(D_Y/J).
\]
This implies that  $\Char(D_Y/J)$ is of dimension $n+1$.
\end{proof}

The $D_X$-module 
$\tilde M/\partial_t \tilde M$ is nothing but the integration of 
the $D_Y$-module $\tilde M$ with respect to $t$, 
and $\tilde M$ is isomorphic to $D_Y/J$ by Theorem \ref{th:J}.
Suppose that $M$ is holonomic on $X_f$. 
Then $\tilde M = D_Y/J$ is a holonomic $D_Y$-module by the theorem above.  
Hence $\tilde M/\partial_t\tilde M$ is also 
a holonomic $D_X$-module. 
In particular, there exists $k_0\in \N$, or else $k_0 = -1$, 
 such that $\tilde M/\partial_t\tilde M$ 
is generated by residue classes 
$[t^j \delta\otimes \iota(u)] = \varphi(u\otimes f^{-j})$  
for $0 \leq j \leq k_0$. 

The $D_X$-module $\tilde M/\partial_t\tilde M$ 
 can be computed by the integration algorithm for $D$-modules 
under the assumption that what is called the $b$-function with respect 
to the weight vector $(0,...,0,1;0,...,0,-1)$ for 
$(x_1,\dots,x_n,t;\partial_{x_1},\dots,\partial_{x_n},\partial_t)$ exists. 
See \cite{OTdeRham}, \cite{SST}, \cite{OakuLecture}. 
This condition is fulfilled if $B_{Z|Y}\otimes_{K[x]}M = D_Y/J$ is holonomic, 
which is the case if $M$ is holonomic on $X_f$.  
Thus we have obtained an algorithm to compute a presentation of $M[f^{-1}]$ 
as left $D_X$-module together with the localization homomorphism $\iota$;  
this algorithm works at least if $M$ is holonomic on $X_f$. 

\begin{definition}\rm
A left $D_X$-module $M$ is said to be {\em $f$-saturated} or 
{\em $f$-torsion free} if the homomorphism $t : M \rightarrow M$ 
is injective. This is equivalent to $H^0_{(f)}(M) = 0$. 
\end{definition}

Let $K$ be algebraically closed and $M$ be a holonomic $D_X$-module. 
Then $\pi^{-1}(\{f=0\})$ contains an irreducible component of $\Char(M)$.  
See e.g.\ Tsai \cite{Tsai} for related topics such as associated primes 
and the Weyl closure of $M$. 
The algorithm above also provides us with a new algorithm to compute 
$H^j_{(f)}(M)$ for $j=0,1$ 
by syzygy computation by means of appropriate Gr\"obner bases of  
a submodule of a free module over $D_X$. 
See the example below for details. 

In what follows, we freely use the notation and the terminology 
introduced in Chapters 1, 2, 4 of \cite{OakuLecture} 
concerning weight vectors, Gr\"obner bases, and $b$-functions. 

\begin{example}\rm
Set $n=1$ and write $x = x_1$. 
Set $\Lsc = D_Xx^s = D_X/D_X(x\partial_x - s)$ with $f = x$. 
Since the $b$-function of $x$ is $s+1$, $\Lsc(\lambda)$ is isomorphic to 
$D_Xx^\lambda$ and hence $x$-saturated if $\lambda \not\in\N$. 
So let us consider $M = \Lsc(0) = D_X/D_X x\partial_x$. 
($M$ is the $D$-module for the Heaviside function $Y(x)$.)
Let $u$ be the residue class of $1$ in $M$. 
The left ideal $J$ of $D_Y$ defined in Theorem \ref{th:J} 
is generated by $tx-1$ and $\partial_x - \partial_tt^2$. 
A Gr\"obner basis of $J$ with respect to a monomial order 
adapted to the weight vector $(1,0;-1,0)$ for $(t,x,\partial_t,\partial_x)$ 
is 
\[
tx-1, \quad x^2\partial_x - \partial_t, \quad
t\partial_t - x\partial_x + 1 = \partial_tt - x\partial_x, \quad 
t^2\partial_t - \partial_x + 2t = \partial_tt^2 - \partial_x.
\]
The $b$-function of $J$ with respect to the weight vector above 
(see Theorem 4.4 of \cite{OakuLecture}) is $s(s+1)$. 
Hence the integration module 
 $\tilde M/\partial_t\tilde M = D_Y/(J + \partial_tD_Y)$ 
is generated by the residue classes $[u]$ and $[tu]$, which correspond 
to $u\otimes 1$ and $u\otimes x^{-1}$ in $M[x^{-1}]$ respectively, 
over $D_X$. 
The fundamental relations among the generators can be 
read off from the Gr\"obner basis above as follows: 
\begin{align*}&
x[tu] - [u] = 0,\quad 
\partial_x[u] = 0, \quad x^2\partial_x[tu] + [u] = 0,
\quad
(x\partial_x+1)[tu] = 0. 
\end{align*}
We translate these relations to those among elements 
$(-1,x)$, $(\partial_x,0)$, $(1,x^2\partial_x)$, $(0,x\partial_x+1)$ of the free module 
$D_X^2$. 
Let $N$ be the left $D_X$-submodule of $D_X^2$ generated by these vectors. 
By using Gr\"obner bases of $N$ with respect to `position-over-term' orders 
(see e.g., Chapter 5 of \cite{CLO2}),  
we can confirm that 
$\Ann_{D_X}[u] = D_X\partial_x$ and $\Ann_{D_X}[tu] = D_X(x\partial_x+1)$.  
Hence $M[x^{-1}]$ is generated by $u\otimes x^{-1}$ and isomorphic to 
$D_X/D_X(x\partial_x+1) \cong K[x,x^{-1}]$ 
with the correspondence $u\otimes x^{-1} \leftrightarrow \overline 1$. 
The image $\iota(M)$ of $\iota : M \rightarrow M[f^{-1}]$ 
is isomorphic to $D_X/D_X\partial_x \cong K[x]$ with the correspondence 
$u\otimes 1 \leftrightarrow \overline 1$.  
Finally we get
\begin{align*}
H^0_{(x)}(M) &= \ker\iota = K[\partial_x]u \cong D_X/D_X x, 
\\
H^1_{(x)}(M) &= D_X[tu]/D_X[u] \cong D_X/D_X x. 
\end{align*}
\end{example}

The following is an example of non-holonomic $M$: 

\begin{example}\rm
Set $n=2$, $x_1 = x$, $x_2=y$, $P = x\partial_x^2+\partial_y$, and  
$M = D_X/D_XP = D_Xu$ as in Example \ref{ex:P}. 
Then $M$ is $x$- and $y$-saturated. 
The localizations of $M$ are 
\begin{align*}
M[x^{-1}] &= D_X(u\otimes x^{-2}) = D_X/D_X(x^2\dx^2+4x\dx+x\dy+2),
\\
M[y^{-1}] &= D_X(u\otimes y^{-1}) = D_X/D_X(xy\dx^2+y\dy+1). 
\end{align*}
The first local cohomology groups are
\begin{align*}
H^1_{(x)}(M) &= D_X[u\otimes x^{-2}] = D_X/(D_Xx^2 + D_Xx\partial_y),
\\
H^1_{(y)}(M) &= D_X[u\otimes y^{-1}] = D_X/D_Xy,
\end{align*}
both of which are not holonomic. 
\end{example}

\section{Algorithm for $M(u,f,s)$ and the $b$-function}

The purpose here is to give algorithms to compute 
$M(u,f,s)$, $M(u,f,\lambda)$, $M\otimes_{K[x]}K[x,f^{-1}]f^\lambda$, 
and the $b$-function $b_{u,f}(s)$ for an arbitrary 
$D_X$-module $M = D_Xu$ that is holonomic on $X_f$, and an arbitrary non-constant 
polynomial $f$. Algorithms for these objects were already given 
in \cite{OakuAdvance} under the additional assumption that $M$ is $f$-saturated. 
We remove this assumption by using the localization algorithm. 

Set $X = K^n$ and $Y = K^{n+1}$ with coordinates 
$x = (x_1,\dots,x_n)$ of $X$ and $(x,t)$ of $Y$. 
Let $f = f(x) \in K[x]$ be a non-constant polynomial and 
let $Z$ be the affine subset $Z = \{(x,t) \mid t = f(x)\}$ of $Y$.
(Note that $Z$ is different from what was defined in the previous section.) 
We regard the local cohomology group 
\[
B_{Z|Y} := H^1_{Z}(K[x,t]) = K[x,t,(t-f)^{-1}]/K[x,t]
\]
as a left $D_Y$-module. For $k\in\N$, let  
\[
\delta^{(k)}(t-f(x)) = \left[ \frac{(-1)^k k!}{(t-f(x))^{k+1}} \right]
\]
be the residue class in $B_{Z|X}$ and denote 
$\delta(t-f) = \delta^{(0)}(t-f)$.  
Then $\delta(t-f)$ satisfies a holonomic system
\[
(t-f)\delta(t-f) = (\partial_{x_i} + f_i\partial_t)\delta(t-f) = 0
\quad (1 \leq i \leq n)
\] 
with $f_i = \partial f/\partial x_i$. 
Hence there exists an isomorphism
\[
B_{Z|Y} \cong
D_Y/(D_Y(t-f) + D_Y(\partial_{x_1} + f_1\partial_t) + \cdots
+ D_Y(\partial_{x_n} + f_n\partial_t)) 
\]
as left $D_Y$-modules. 
In particular, $\delta^{(k)}(t-f)$ ($k\in\N$) constitute 
a free basis of $B_{Z|Y}$ over $K[x]$. 

Following B.\ Malgrange, let us give $\Lsc = D_X[s]f^s$ 
a structure of left $D_Y$-module by
\[
t(a(x,s)f^{-k}f^s) = a(x,s+1)f^{-k+1}f^s, 
\qquad 
\partial_t(a(x,s)f^{-k}f^s) = -sa(x,s-1)f^{-k-1}f^s
\]
for $a(x,s) \in K[x,s]$ and $k \in \N$. 
The actions of $t$ and $\partial_t$ on $\Lsc$ satisfy 
$[\partial_t,t] = 1$, and they commute with $x_i$, $\partial_{x_i}$. 
Hence the definition above extends to the action of $D_Y$ on $\Lsc$. 
In particular, $\partial_t t f^s = -sf^s$ holds, which will 
play an important role in what follows. 

With respect to this action, $B_{Z|Y}$ can be regarded as a left 
$D_Y$-submodule of $\Lsc$ by 
identifying $\delta^{(k)}(t-f)$ with 
$(-1)^ks(s-1)\cdots(s-k+1)f^{s-k}$ in $\Lsc$. 
In fact, we have 
\begin{align*}&
t\delta^{(k)}(t-f) = f\delta^{(k)}(t-f) - k\delta^{(k-1)}(t-f), 
\\& 
t((-1)^ks(s-1)\cdots(s-k+1)f^{s-k}) 
= (-1)^k(s+1)s(s-1)\cdots(s-k+2)f^{s-k+1}
\\&=
f(-1)^ks(s-1)\cdots(s-k+1)f^{s-k} 
- (-1)^{k-1} ks(s-1)\cdots(s-k+2)f^{s-k+1}
\end{align*}
since $(t-f)\delta^{(k)}(t-f) = -k\delta^{(k-1)}(t-f)$, and 
\begin{align*}&
\partial_t\delta^{(k)}(t-f) = \delta^{(k+1)}(t-f),
\\&
\partial_t((-1)^ks(s-1)\cdots(s-k+1)f^{s-k}) 
= (-1)^{k+1}s(s-1)\cdots(s-k)f^{s-k-1}. 
\end{align*}

For a left $D_X$-module $M = D_Xu$, consider the tensor product 
$M\otimes_{K[x]}B_{Z|Y}$, which is a left $D_Y$-module. 

\begin{definition}\rm
Set $\vartheta'_i := \partial_{x_i} + f_i\partial_t$ 
for $i=1,\dots,n$.  
Define a ring homomorphism $\tau'$ from $D_X$ to $D_Y$ by
\[
\tau' : D_X \ni P(x,\partial_x) \longmapsto 
P(x,\vartheta'_1,\dots,\vartheta'_n)
\in D_Y. 
\]
Since $\vartheta'_1,\dots,\vartheta'_n$ commute with each other, 
and $[\vartheta'_i,x_j] = \delta_{ij}$, this substitution is a well-defined 
ring homomorphism. 
\end{definition}

Since $\vartheta'_i\delta(t-f) = 0$ for $1 \leq i \leq n$, 
\[
 \tau'(P)(v \otimes \delta(t-f)) = Pv \otimes \delta(t-f)
\]
holds for any $v \in M$ and $P \in D_X$. Hence we have
\[
Pu \otimes \partial_t^k\delta(t-f) 
= \partial_t^k(Pu \otimes \delta(t-f))
= \partial_t^k\tau'(P)(u \otimes \delta(t-f)). 
\]
This proves

\begin{lemma}
If $M = D_Xu$, then 
the left $D_Y$-module $M\otimes_{K[x]} B_{Z|Y}$ 
is generated by $u\otimes \delta(t-f)$. 
\end{lemma}

\begin{theorem}\label{th:J'}
Let $M = D_Xu$ be a left $D_X$-module generated by $u$ and $f \in K[x]$.
Let $G$ be a finite set of generators of $I := \Ann_{D_X}u$ and 
let $J$ be the left ideal of $D_Y$ generated by 
$\{\tau'(P) \mid P \in G\} \cup \{t-f\}$. Then $J$ coincides with 
the annihilator $\Ann_{D_Y} (u\otimes \delta(t-f))$.
Moreover, if $M$ is a holonomic $D_X$-module, 
then $M \otimes_{K[x]}B_{Z|Y}$ is a holonomic $D_Y$-module.  
\end{theorem}

The first part of this proposition was proved by Walther \cite{Walther1}. 
The proof is almost the same as the proof of Theorem \ref{th:J}. 
The last assertion can be proved in the same way as Theorem \ref{th:Jholonomic}.

Thus we have an algorithm to compute $M\otimes_{K[x]}B_{Z|Y}$, 
which was already given in \cite{OakuAdvance}.  
The inclusion $B_{Z|X} \subset \Lsc$ induces a homomorphism 
\[
\psi : M\otimes_{K[x]}B_{Z|Y}  \longrightarrow M\otimes_{K[x]}\Lsc
\]
of left $D_Y$-modules. 
The image of $\psi$ coincides with $M(u,f,s)$.

Our main aim is to compute 
the $D_X[s]$-submodule $M(u,f,s) = D_X[s](u\otimes f^s)$ 
of $M\otimes_{K[x]}\Lsc$. 
The following lemma was proved in \cite{OakuAdvance} as Proposition 6.13. 

\begin{lemma}\label{lemma:psi}
The homomorphism $\psi$ above is injective if and only if 
$H_{(f)}^0(M) = 0$; i.e., $M$ is $f$-saturated, 
\end{lemma}

\begin{proof}
An arbitrary element of $M\otimes_{K[x]}B_{Z|Y}$ is expressed uniquely 
as 
\[
 w =  \sum_{j=0}^k v_j \otimes \delta^{(j)}(t-f)
\]
with $k \in \N$ and $v_j \in M$. 
Then 
\[
\psi(w) = \sum_{j=0}^k (-1)^j v_j\otimes (s(s-1)\cdots (s-j+1)f^{-j} f^{s}) 
\]
vanishes if and only if each $v_j\otimes f^{-j}$ vanishes in 
$M\otimes_{K[x]}K[x,f^{-1}]$, which is equivalent to 
$v_j \in H^0_{(f)}(M)$. 
\end{proof}

\begin{lemma}
Let $M = D_Xu$  be a left $D_X$-module generated by $u$ and $f \in K[x]$ be 
a non-constant polynomial. Let 
$\iota : M \rightarrow M[f^{-1}]$ be the canonical homomorphsm. 
Then $\iota$ induces isomorphisms
\[
M\otimes_{K[x]}\Lsc \stackrel{\sim}{\longrightarrow} 
\iota(M)\otimes_{K[x]}\Lsc,
\qquad
M(u,f,s) \stackrel{\sim}{\longrightarrow} \iota(M)(u,f,s)
\] 
of left $D_X[s]$-modules.  
\end{lemma}

\begin{proof}
Since $\Lsc$ is isomorphic to $K[x,f^{-1},s]$ as a $K[x,s]$-module, 
we have only to show that the natural homomorphism 
$M[f^{-1}] \rightarrow \iota(M)[f^{-1}]$ is an isomorphism. 
For any $v\in M$, $v\otimes 1 = 0$ holds in $M[f^{-1}]$ if and only if 
$f^kv = 0$ in $M$ with some $k\in\N$. This is equivalent to 
$\iota(v)\otimes 1 = 0$ in $\iota(M)[f^{-1}]$. 
Hence the first homomorphism is an isomorphism. 
This implies the injectivity of the second homomorphism, 
the surjectivity of which is obvious by the definition. 
\end{proof}

Summing up we obtain

\begin{algorithm}[Computing $M(u,f,s)$ and $M(u,f,\lambda)$]\rm

Input: $M = D_Xu = D_X/I$ with a finite set of generators of $I$, 
a non-constant $f \in K[x]$, and $\lambda \in K$. 

\noindent
Output: presentation of $M(u,f,s)$, i.e., $\Ann_{D_X[s]}(u\otimes f^s)$, 
and of $M(u,f,\lambda)$.  
\begin{enumerate}
\item
Compute $\iota(M) = D_X/\Ann_{D_Y}\iota(u)$ by the localization algorithm.
\item
Compute $J = \Ann_{D_Y} (\iota(u)\otimes \delta(t-f))$ by using 
Theorem \ref{th:J'}. 
\item
Compute $J \cap D_X[s]$, which is equal to $\Ann_{D_X[s]} (u\otimes f^s)$.
\item
The substitution $s = \lambda$ for generators of $J \cap D_X[s]$ 
gives a set of generators of $\Ann_{D_X}(u\otimes f^s)|_{s=\lambda}$.   
\end{enumerate}
\end{algorithm}

\begin{proposition}
Let $M = D_Xu$ be a left $D_X$-module generated by $u$ and $f \in K[x]$ 
a non-constant polynomial. 
Then the $b$-function $b_{u,f}(x)$ exists if and only if there exists a 
nonzero polynomial $b(s) \in K[s]$ such that 
\begin{equation}\label{eq:tP}
b(-\partial_tt)(u\otimes\delta(t-f)) \in tD_X[t\partial_t](u\otimes\delta(t-f))
= D_X[t\partial_t]t(u\otimes\delta(t-f))
= D_X[t\partial_t]f(u\otimes\delta(t-f)). 
\end{equation}
If $M$ is $f$-saturated and such $b(s)$ exists, then  $b_{u,f}(s)$ is 
the monic polynomial of the minimum degree among such $b(s)$.  
\end{proposition}

\begin{proof}
Assume that (\ref{eq:tP}) holds. Then there exists $P(s) \in D_X[s]$ 
such that 
\[
b(-\partial_tt)(u\otimes\delta(t-f)) = P(-\partial_tt)f(u\otimes \delta(t-f)). 
\]
Applying the homomorphism $\psi$ we get 
$b(s)(u\otimes f^s) = P(s)(u\otimes f^{s+1})$. Hence $b_{u,f}(s)$ exists and 
divides $b(s)$. 

On the other hand, assume that there exist nonzero $b(s) \in K[s]$ and 
 $P(s) \in D_X[s]$ such that 
$b(s)(u\otimes f^s) = P(s)(u\otimes f^{s+1})$ holds in $M\otimes_{K[x]}\Lsc$. 
Then as is seen by the proof of Lemma \ref{lemma:psi}, there exists $k \in \N$ 
such that
\[
f^kb(-\partial_tt)(u\otimes\delta(t-f)) 
= f^kP(-\partial_tt)f(u\otimes \delta(t-f)). 
\]
Since 
\[
f^kb(-\partial_tt)(u\otimes\delta(t-f)) = 
b(-\partial_tt)f^k(u\otimes\delta(t-f)) 
= b(-\partial_tt)t^k(u\otimes\delta(t-f)) 
\]
holds, we get
\[
\partial_t^kb(-\partial_tt)t^k(u\otimes\delta(t-f)) 
= \partial_t^kf^kP(-\partial_tt)f(u\otimes\delta(t-f)) 
\in D_X[t\partial_t]f(u\otimes\delta(t-f)). 
\]
This completes the proof because there exists $c(s) \in K[x]$ such that
$\partial_t^kb(-\partial_tt)t^k = c(-\partial_tt)$. 
\end{proof}

Now we obtain an algorithm to determine whether the $b$-function exists and 
to compute it if it does: 

\begin{algorithm}[Computing the $b$-function $b_{u,f}(s)$]\rm
\mbox{}\\
Input: $M = D_Xu = D_X/I$ with a finite set of generators of $I$ and 
a non-constant $f \in K[x]$. 

\noindent
Output: the $b$-function $b_{u,f}(s)$ if it exists. `No' if it does not exist.  
\begin{enumerate}
\item
Compute $J' := \Ann_{D_Y}(u\otimes \delta(t-f))$ by using Theorem \ref{th:J'}. 
\item
Compute $I' = (J' + D_X[s]f) \cap K[s]$ by elimination. 
If $I' \neq \{0\}$, then there exists $b_{u,f}(s)$. 
Otherwise, the $b$-function does not exist; output `No' and quit.  
\item
Compute a set of generators of 
$J := \Ann_{D_X[s]} (u\otimes f^s)$ by the preceding algorithm. 
\item
Compute $I(u,f) = (J + D_X[s]f) \cap K[s]$ by elimination. 
The  monic  generator of $I(u,f)$ is  $b_{u,f}(s)$. 
\end{enumerate}
\end{algorithm}

\begin{algorithm}[Computing $D_X(u\otimes f^\lambda)$] 
\rm
\mbox{}\\
Input: $M = D_Xu = D_X/I$ with a finite set of generators of $I$, 
a non-constant $f \in K[x]$, and $\lambda \in K$. 

\noindent
Output: presentation of $D_X(u\otimes f^\lambda)$, 
i.e., $\Ann_{D_X}(u\otimes f^\lambda)$ if $b_{u,f}(s)$ exists. 
\begin{enumerate}
\item
Compute $M(u,f,s)$ and $b_{u,f}(s)$ by preceding algorithms.  
Quit if $b_{u,f}(s)$ does not exist. 
\item
Let $k_0$ be the maximum nonzero integer, if any, 
such that $b_{u,f}(\lambda - k_0) = 0$. 
If there is no such $k_0$, then set $k_0 = 0$. 
\item
Compute $I := \Ann_{D_X}(u\otimes f^{\lambda-k_0})$ 
from $\Ann_{D_X[s]}(u\otimes f^s)$ by substitution $s = \lambda - k_0$. 
\item
Compute the left ideal 
\[
I : f^{k_0} := \{ P \in D_X \mid Pf^{k_0} \in I\}
\]
by an appropriate Gr\"obner basis. 
Then $I : f^{k_0} = \Ann_{D_X} (u\otimes f^\lambda)$.  
\end{enumerate}
\end{algorithm}

\begin{example}\rm
Set $n=2$, $x_1 = x$, $x_2=y$, $P = x\partial_x^2+\partial_y$, and  
$M = D_X/D_XP = D_Xu$ as in Example \ref{ex:P}, where it is shown that 
$b_{u,x}(s) = (s+1)(s+2)$ and $b_{u,y}(s) = s+1$. We have
\begin{align*}
M(u,x,s) &:= D_x[s](u\otimes x^s) = D_X[s]/D_X[s](x^2\dx^2-2sx\dx+x\dy+s^2+s)
\\
M(u,y,s) &:= D_x[s](u\otimes y^s) = D_X[s]/D_X[s](xy\dx^2+y\dy-s).
\end{align*}
\end{example}

\begin{example}\label{ex:L16}\rm
Set $n=2$ and write $x_1 = x$, $x_2 = y$. 
Set $f(x,y) = x^3-y^2$ and 
\[
\Lsc = D_X[s]f^s, \qquad
M := \Lsc\Bigl(\frac16\Bigr) = \Lsc/\Bigl(s-\frac16\Bigr)\Lsc. 
\]
Then $M[f^{-1}]$ is generated by 
$f^{-1}\iota((f^s)|_{s=1/6})$
and isomorphic to $D_X/J$ with the left ideal $J$ of $D_X$ generated by
\[
2x\dx + 3y\dy + 5, \quad 2y\dx+3x^2\dy,\quad
4y\dx^2-9xy\dy^2-12x\dy. 
\]
In fact, $M[f^{-1}]$ is isomorphic to $K[x,f^{-1}]f^{1/6}$, 
and $\iota(M)$ to $D_Xf^{1/6}$.  

The homomorphic image $\iota(M)$ of $M$ in 
$M[f^{-1}]$ is given by
\begin{align*}
\iota(M) &= D_X\iota((f^s)|_{s=1/6}) = D_X/\tilde I,
\\
\tilde I &= D_X(2x\partial_x + 3y\partial_y-1) + D_X(2y\partial_x + 3x^2\partial_y) 
+ D_X(8\partial_x^3 + 27y\partial_y^3 + 9\partial_y^2), 
\end{align*}
while
\[
 M = D_X(f^s|_{s=1/6}) = D_X/I, 
\qquad 
I = D_X(2x\partial_x + 3y\partial_y-1) + D_X(2y\partial_x + 3x^2\partial_y).  
\]
By syzygy computation, we get
\begin{align*}
H^0_{(f)}(M) &\cong D_X/(D_Xx+D_Xy) \cong H^2_{(x,y)}(K[x,y]),
\\
H^1_{(f)}(M) &= M[f^{-1}]/\iota(M) \cong D_X/(D_Xx+D_Xy) \cong H^2_{(x,y)}(K[x,y]). 
\end{align*}
Set $u = f^s|_{s=1/6}$. Then the $b$-function for $u$ and $f$ is 
\[
b_{u,f}(s) = (s+1)\bigl(s+\frac43\bigr)\bigl(s+\frac76\bigr) 
= b_f\bigl(s + \frac16\bigr),
\]
where $b_f(s) = (s+1)(s+5/6)(s+7/6)$ is the $b$-function of $f$. 
\end{example}

\begin{example}\rm
Set $n=2$ and write $x_1 = x$, $x_2 = y$. 
Let $M = H^1_{(xy)}(K[x,y])$ be the first local cohomology group supported by 
$xy = 0$. Let $\iota : M \rightarrow M[x^{-1}]$ be the canonical homomorphism. 
Then we have 
\begin{align*}
 M[x^{-1}] &= \iota(M) = D_X\iota([(xy)^{-1}]) = D_X/(D_X(x\partial_x + 1) + D_X y),
\\
 H^0_{(x)}(M)& \cong D_X/(D_Xx + D_X\partial_y) \cong H^1_{(x)}(K[x,y]), 
\quad 
H^1_{(x)}(M) = 0. 
\end{align*}
The $b$-function of $u := [(xy)^{-1}]$ and $x$ is 
$b_{u,x}(s) = s$. The module $M(u,x,s)$ is 
\[
M(u,x,s) = D_X[s]/(D_X[s](x\dx-s+1) + D_X[s]y).
\]
\end{example}

\begin{example}\rm
Set $n=3$ and write $x_1 = x$, $x_2 = y$, $x_3 = z$. 
Set  $f= x^3-y^2z^2$, $g = x^3-y^2$, and 
$M = H^1_{(f)}(K[x,y])$. 
Note that $f$ has non-isolated singularities, and
$f$ and $g$ are not of complete intersection. 
$M$ is saturated with respect to $x$, $y$, $g$, and
the $b$-functions with $u := [f^{-1}]$ are 
\begin{align*}
&
b_{u,x} = (s+1)^2\bigl(s - \frac12\bigr)^2, 
\quad
b_{u,y} = (s+1)\bigl(s- \frac13\bigr)\bigl(s+ \frac13\bigr), 
\\&
b_{u,g} = (s+1)^2\bigl(s + \frac12\bigr)
\bigl(s + \frac13\bigr)\bigl(s+ \frac23\bigr)
\bigl(s - \frac16\bigr)\bigl(s+ \frac16\bigr), 
\end{align*}
while the $b$-function of $f$ is 
\[
(s+1)\bigl(s+\frac43\bigr)\bigl(s+\frac53\bigr)
\bigl(s+\frac56\bigr)\bigl(s+\frac76\bigr). 
\]
The first local cohomology groups are
\begin{align*}
H^1_{(x)}(M) &\cong D_X/(D_Xx + D_X(y*dy+2) + D_Xy^2z^2 + D_X(z\dz+2)), 
\\
H^1_{(y)}(M) &\cong D_X/(D_Xx^3 + D_X(x\dx+3) + D_Xy + D_X\dz,
\\
H^1_{(g)}(M) &\cong D_X/(D_Xg + D_Xy^2(z^2-1) + D_XP_1 + D_XP_2 + D_XP_3
\end{align*}
with 
\[
P_1 = 2y^2\dx+3yx^2\dy-6x^2,
\quad P_2 = (z^2-1)y\dy+2z^2-2, \quad P_3 = 2x\dx+3y\dy+12.
\]
These cohomology groups are isomorphic to the second local cohomology groups of 
$K[x,y,z]$ supported by the ideals $(f,x)$, $(f,y)$, $(f,g)$ respectively,
which can be also computed by an algorithm of \cite{Walther1} or \cite{OTalgDmod}.
The multiplicities (see the next section) of these modules are $3$, $1$, and $13$ 
respectively. In fact, $H^1_{(y)}(M) = H^2_{(f,y)}(K[x,y,z])$ is isomorphic to 
$H^2_{(x,y)}(K[x,y,z])$.  
\end{example}

\section{Length and multiplicity of $D$-modules}

W set $X = K^n$ as in the preceding sections. 
First let us recall basic facts about the length and the multiplicity 
of a left $D_X$-module following J.~Bernstein (\cite{Bernstein1},\cite{Bernstein2}). 
Let $M$ be a finitely generated left $D_X$-module. 
A composition series of $M$ of length $k$ is a sequence 
\[
M = M_0 \supset M_1 \supset \cdots \supset M_k = 0
\]
of left $D_n$-submodules such that 
$M_i/M_{i-1}$ is a nonzero simple left $D_X$-module 
(i.e.\ having no proper left $D_X$-submodule other than $0$)  
for $i=1,\dots k$. 
The {\em length} of $M$, which we denote by $\length M$, is 
the least length of composition series (if any) of $M$. 
If there is no composition series, the length of $M$ is defined 
to be infinite.  
The length is additive in the sense that if 
\[
 0 \longrightarrow N \longrightarrow M \longrightarrow L \longrightarrow 0
\]
is an exact sequence of left $D_X$-modules of finite length, 
then $\length M = \length N + \length L$ holds. 

For each integer $k$, set
\[
F_k(D_X) = \Bigl\{\sum_{|\alpha| + |\beta| \leq k}
a_{\alpha\beta}x^\alpha\partial^\beta \mid a_{\alpha\beta}\in K\}\Bigr\}. 
\]
In particular, we have $F_k(D_X) = 0$ for $k < 0$ and $F_0(D_X) = K$. 
The filtration $\{F_k(D_X)\}_{k\in\Z}$ 
is called the Bernstein filtration on $D_X$. 

Let $M$ be a finitely generated left $D_X$-module. 
A family $\{F_k(M)\}_{k\in \Z}$ 
of $K$-subspaces of $M$ is called a {\em Bernstein filtration} on $M$ if it 
satisfies
\begin{enumerate}
\item 
$F_k(M) \subset F_{k+1}(M)\quad (\forall k\in \Z), 
\qquad \bigcup_{k \in \Z} F_k(M) = M$
\item
$F_j(D_X)F_k(M) \subset F_{j+k}(M)\quad (\forall j,k\in \Z)$
\end{enumerate} 

Moreover, $\{F_k(M)\}$ is called a {\em good Bernstein filtration} if
there exist $u_i \in F_{k_i}(M)$ ($i=1,\dots,m$) such that 
\[
F_k(M) = F_{k-k_1}(D_X)u_1 + \cdots + F_{k-k_m}(D_X)u_m \qquad (\forall k\in \Z). 
\]
If $\{F_k(M)\}$ is a good Bernstein filtration, then each $F_k(M)$ is
a finite dimensional vector space over $K$ and $F_k(M) = 0$ for 
$k \ll 0$ (see e.g., 1.3 of \cite{OakuLecture}). 

Let $\{F_k(M)\}$ be a good Bernstein filtration on $M$. 
Then there exists a polynomial  
$p(T) = c_dT^d + c_{d-1}T^{d-1} + \cdots + c_0\in \Q[T]$ such that 
\[
\dim_K F_k(M) = p(k) \quad (k \gg 0)
\]
and $d!c_d$ is a positive integer.  
We call $p(T)$ the Hilbert polynomial of $M$ with respect to 
the filtration $\{F_k(M)\}$. 
The leading term of $p(T)$ does not depend on the choice of 
a good Bernstein filtration $\{F_k(M)\}$.  
The degree $d$ of the Hilbert polynomial $p(T)$ is 
called the {\em dimension} of $M$ and denoted  $\dim M$. 
The {\em multiplicity} of $M$, denoted $\mult M$  is defined to be the positive integer $d!c_d$.
The dimension and the multiplicity are invariants of a finitely generated 
left $D_X$-module.

If $M \neq 0$, then the dimension of $M$ is not less than $n$ 
(Bernstein's inequality). 
By definition,  $M$ is holonomic if $M=0$ or $\dim M = n$. 
If $M$ is a holonomic left $D_X$-module, we have an inequality
$\length M \leq \mult M$ and hence $M$ is of finite length in particular. 
Moreover, the multiplicity 
is additive for holonomic left $D_X$-modules. 

We can compute the dimension and the multiplicity of a given finitely generated 
(not necessarily holonomic) $D_X$-module by using a Gr\"obner basis with 
respect to a term order compatible with the Bernstein filtration. 

\begin{example}\rm
Let $M$ be the $D_X$-module with $X = K^2$ defined in Example \ref{ex:L16}. 
We get exact sequences
\begin{align*}
& 0 \longrightarrow H^0_{(f)}(M) \longrightarrow M \longrightarrow 
\iota(M) \longrightarrow 0
\\&
 0 \longrightarrow H^0_{(f)}(M) \longrightarrow M \longrightarrow 
M[f^{-1}] \longrightarrow H^1_{(f)}(M) \longrightarrow 0
\end{align*}
with $H^0_{(f)}(M) \cong H^2_{(x,y)}(K[x,y]) \cong H^1_{(f)}(M)$.
We have 
\[
\mult M = \mult M[f^{-1}] = 6, \quad \mult\iota(M) = 5, \quad
\mult H^0_{(f)}(M) = \mult H^1_{(f)}(M) = 1.
\]
\end{example}

The following two propositions are easy and  should be well-known. 

\begin{proposition}
Let $f \in K[x]$ be a non-constant polynomial. Then 
the multiplicity of the left $D_X$-module $K[x,f^{-1}]$ is at most $(\deg f+1)^n$. 
\end{proposition}

\begin{proof}
Let $d$ be the degree of $f$.  
\[
F_k(K[x,f^{-1}]) := \left\{\frac{a}{f^{k+1}}\mid 
a \in K[x_1,\dots,x_n],\, \deg a \leq (d+1)k \right\} 
\quad(k \in \Z), 
\]
is a (not necessarily good) Bernstein filtration 
on $M$ with 
\[
\dim_K F_k(K[x,f^{-1}]) = \binom{n+(d+1)k}{n} .
\]
This implies $\dim K[x,f^{-1}] = n$ and $\mult M \leq (d+1)^n$. 
\end{proof}

\begin{proposition}
Let  $n=1$ and $f \in K[x] = K[x_1]$ be non-constant square free. 
Then one has 
$\,\,\mult K[x,f^{-1}] = \deg f +1$. 
\end{proposition}

\begin{proof} 
Set 
$M := H_{(f)}^1(K[x])$. Then $M$ is isomorhpic to $D_X/D_Xf$ since $f$ 
is square-free. Hence 
$F_k(M) := F_k(D_1)[f^{-1}] \cong F_k(D_1)/F_{k-d}(D_2)f$ 
with $d := \deg f$ constitute a good Bernstein filtration on $M$. Since
\begin{align*}
\dim F_k(M) &= \dim F_k(D_1) - \dim F_{k-d}(D_1)
\\&= 
\binom{k+2}{2} - \binom{k-d+2}{2} = dk - \frac12 (d-1)(d-2),
\end{align*}
hods for $k \geq d$, the multiplicity of $M$ is $d$.  
\end{proof}

We shall give two examples in two variables. 

\begin{proposition}
Set $X = K^2$ and write $x_1 = x$, $x_2 = y$. 
Set $f = x^m + y^l$ with positive integers $l,m$. 
Then the multiplicity of $K[x,y,f^{-1}]$ equals $2\max\{l,m\}$. 
\end{proposition}

\begin{proof}
We may assume $m \leq l$. 
Set $M := H_{(f)}^1(K[x,y])$. 
Since the $b$-function $b_f(s)$ of $f$ does not have any negative 
integer $\leq -2$ as a root (see e.g., 6.4 of \cite{KashiwaraAlg}), 
$M$ is generated by  $u:=[f^{-1}] \in M$ over $D_X$.  
The annihilator $\Ann_{D_X} u$ is generated by 
\begin{align*}&
f,\quad
E := lx\partial_x + my\partial_y + ml, 
\quad P := ly^{l-1}\partial_x - mx^{m-1}\partial_y
\end{align*}
(see also 6.4 of \cite{KashiwaraAlg}).  
A Gr\"obner basis of $\Ann_D[f^{-1}]$ with respect to  
a total-degree reverse lexicographic order $\prec$ such that 
$x \succ y \succ \xi \succ \eta$ is $G = \{f,E,P\}$, 
where $\xi$ and $\eta$ are the commutative variables 
corresponding to $\partial_x$ and $\partial_y$ respectively.   
In fact, in case $m<l$ the $S$-pairs (see Chapter 2 of \cite{OakuLecture}) are 
divisible by $G$: 
\begin{align*}&
\spoly_\prec(f,E) = lx\partial_xf - y^lE = x^mE - my\partial_y f,
\\&
\spoly_\prec(f,P) = l\partial_xf - yP = x^{m-1}E, 
\quad
\spoly_\prec(E,P) = y^{l-1}E - xP = m\partial_y f.
\end{align*}
The initial monomials of the Gr\"obner basis $G$ are
$\init_\prec(f) = y^l$, $\init_\prec(E) = x\xi$, $\init_\prec(P) = y^{l-1}\xi$.  
Hence for $k \geq l$ we obtain 
\begin{align*}&
\dim_K F_k(D_X)/(\Ann_{D_X} [f^{-1}] \cap F_k(D_X))
\\&
= \sharp (\{x^iy^j\xi^\mu\eta^\nu \mid i+j+\mu+\nu \leq k\}
\setminus \langle y^l, x\xi, y^{l-1}\xi \rangle)
\\&=
\sharp\{x^iy^j\eta^\nu \mid i+j+\nu \leq k,\, j \leq l-1\}
+
\sharp\{y^j\xi^\mu\eta^\nu \mid j+\mu+\nu \leq k,\, j \leq l-2,\,\mu \geq 1\}
\\&
= \sum_{j=0}^{l-1}\binom{2+k-j}{2} + \sum_{j=0}^{l-2}\binom{2+k-j-1}{2}
= \frac{2l-1}{2}k^2 + \cdots. 
\end{align*}

On the other hand, in case $m=l$ we have 
\begin{align*}&
\spoly_\prec(f,E) = l\partial_xf - x^{l-1}E = yP, \quad
\spoly_\prec(f,P) = ly^{l-1}\partial_xf - x^lP = y^l P + lx^{l-1}\partial_yf
\\&
\spoly_\prec(E,P) = y^{l-1}E - xP = l\partial_y f.
\end{align*}
The initial monomials are
$ \init_\prec(f) = x^l$, $\init_\prec(E) = x\xi$, $\init_\prec(P) = y^{l-1}\xi$. 
(Note that $y^{l-1}\xi \succ x^{l-1}\eta$ holds.)
Hence for $k \geq l$ we obtain
\begin{align*}&
\dim_K F_k(D_X)/(\Ann_{D_X} [f^{-1}] \cap F_k(D_X))
\\&
= \sharp (\{x^iy^j\xi^\mu\eta^\nu \mid i+j+\mu+\nu \leq k\}
\setminus \langle x^l, x\xi, y^{l-1}\xi \rangle)
\\&=
\sharp\{x^iy^j\eta^\nu \mid i+j+\nu \leq k,\, i \leq l-1\}
+
\sharp\{y^j\xi^\mu\eta^\nu \mid j+\mu+\nu \leq k,\, j \leq l-2,\,\mu \geq 1\}
\\&
= \sum_{i=0}^{l-1}\binom{2+k-i}{2} + \sum_{j=0}^{l-2}\binom{2+k-j-1}{2}
= \frac{2l-1}{2}k^2 + \cdots.
\end{align*}
Hence the multiplicity of $M$ is $2l$ in both cases. This proves the assertion. 
\end{proof}

\begin{proposition}
Set $X=K^2$ with $x_1 = x$ and $x_2 = y$. 
Set $f = x^m + y^l +1$ with positive integers $l,m$. 
Then the multiplicity of $K[x,y,f^{-1}]$ equals  
$lm + |l-m| + 1$.  
\end{proposition}

\begin{proof}
We may assume $m \leq l$. 
Set $M := H_{(f)}^1(K[x,y])$. 
Since the curve $f=0$ is non-singular, 
the $b$-function is $b_f(s) = s+1$. 
Hence $M$ is generated by $u:= [f^{-1}]$.  
The annihilator $\Ann_{D_X}u$ is generated by
$f$ and $P := ly^{l-1}\partial_x - mx^{m-1}\partial_y$ 
since $f=0$ is non-singular.

In case $l=m$, 
$G = \{f,P\}$ is a 
Gr\"obner basis of $\Ann_{D_X}[f^{-1}]$ with respect to a 
total-degree reverse lexicographic order $\prec$ such that 
$x \succ y \succ \xi \succ \eta$. 
In fact, we have
\[
\spoly_\prec(f,P) = ly^{l-1}\partial_xf - x^lP 
= ly^{l-1}\partial_x f + x^lP. 
\]
Since $\init_\prec(f) = x^l$ and $\init_\prec(P) = y^{l-1}\xi$, we have 
for $k \geq 2l$ 
\begin{align*}&
\dim_K F_k(D_X)/(\Ann_{D_X} [f^{-1}] \cap F_k(D_X))
\\&
= \sharp (\{x^iy^j\xi^\mu\eta^\nu \mid i+j+\mu+\nu \leq N\}
\setminus \langle x^l, y^{l-1}\xi \rangle)
\\&
= \sharp \{x^iy^j\eta^\nu \mid i+j+\nu \leq k,\, i \leq l-1\}
\\&\quad
+ \sharp \{x^iy^j\xi^\mu\eta^\nu \mid i+j+\mu+\nu \leq k,\, i \leq l-1,\,
0 \leq j \leq l-2,\,\mu\geq 1\}
\\&
= \sum_{i=0}^{l-1}\binom{2+k-i}{2}
+ \sum_{i=0}^{l-1}\sum_{j=0}^{l-2}\binom{2+k-i-j-1}{2}
= \frac{l^2}{2}k^2 + \cdots.
\end{align*}
In case $m < l$, the  Gr\"obner basis of $\Ann_D[f^{-1}]$ with respect to 
the same order as above 
is 
$G =\{f,P,Q\}$ with
\[
Q := l(x^m+1)\partial_x + mx^{m-1}y\partial_y + ml x^{m-1}.
\]
In fact, we have 
\begin{align*}&
\spoly_\prec(f,P) = l\partial_x f - yP = Q,
\\&
\spoly_\prec(f,Q) = lx^m\partial_x f - y^lQ
=  -mx^{m-1}y\partial_yf - yP + x^mQ,
\\&
\spoly_\prec(P,Q) = x^mP - y^{l-1}Q
= -mx^{m-1}\partial_y f - P. 
\end{align*}
Since $\init_\prec(f) = y^l$, $\init_\prec(P) = y^{l-1}\xi$, 
$\init_\prec(Q) = x^m\xi$, we have for $k \geq l+m$, 
\begin{align*}&
\dim_K F_k(D_X)/(\Ann_{D_X} [f^{-1}] \cap F_k(D_X))
\\&
= \sharp (\{x^iy^j\xi^\mu\eta^\nu \mid i+j+\mu+\nu \leq k\}
\setminus \langle y^l, y^{l-1}\xi, x^m\xi \rangle)
\\&
= \sharp \{x^iy^j\eta^\nu \mid i+j+\nu \leq k,\, i \leq l-1\}
\\&
+ \sharp \{x^iy^j\xi^\mu\eta^\nu \mid i+j+\mu+\nu \leq k,\, i \leq m-1,\,
j \leq l-2,\,\mu\geq 1\}
\\&
= \sum_{i=0}^{l-1}\binom{2+k-i}{2}
+ \sum_{i=0}^{m-1}\sum_{j=0}^{l-2}\binom{2+k-i-j-1}{2}
= \frac{l+m(l-1)}{2}k^2 + \cdots.
\end{align*}
Hence the multiplicity of $M$ is $l+m(l-1) = ml + l-m$. 
\end{proof}

Now let us resume the study on  $M(u,f,s)$ for a $D_X$-module $M = D_Xu$ and 
a polynomial $f$. 

\begin{lemma}
Let $M = D_Xu$ be a left $D_X$-module generated by $u$. 
For any $\lambda \in K$, the endomorphism of $M(u,f,s)$ defined by
$s-\lambda$ is injective. Hence the sequence
\[
0 \longrightarrow M(u,f,s) \stackrel{s-\lambda}{\longrightarrow} 
M(u,f,s) \longrightarrow M(u,f,\lambda) \longrightarrow 0 
\]
of left $D_X$-modules is exact. 
\end{lemma}

\begin{proof}
We may assume that $M$ is $f$-saturated as was seen in the previous section. 
Then the homomorphism 
$\psi : M \otimes_{K[x]}B_{Z|Y} \rightarrow M \otimes_{K[x]}\Lsc$ 
is injective by Lemma \ref{lemma:psi}. 

Hence we have only to show that 
$s-\lambda = -\partial_tt-\lambda$ is an injective endomorphism of 
$M \otimes_{K[x]}B_{Z|Y}$.  
Let 
\[
v = \sum_{j=0}^k v_j \otimes \delta^{(j)}(t-f)
\]
be an arbitrary element of $M \otimes_{K[x]}B_{Z|Y}$ with 
$k \in \N$ and $v_j \in M$. Then we get 
\begin{align*}
(s-\lambda)v &= -\sum_{j=0}^k v_j\otimes(t\partial_t + \lambda+1)
\delta^{(j)}(t-f) 
\\&=
-\sum_{j=0}^k v_j\otimes(f\delta^{(j+1)}(t-f) + (\lambda - j)\delta^{(j)}(t-f))
\\&
=-\lambda v_0 \otimes \delta(t-f)
 -\sum_{j=1}^k (fv_{j-1} + (\lambda- j)v_j)\otimes \delta^{(j)}(t-f) 
- fv_k\otimes \delta^{(k+1)}(t-f). 
\end{align*}
Thus $(s-\lambda)v=0$ is equivalent to 
\[
\lambda v_0 = fv_k = fv_{j-1} + (\lambda-j)v_j = 0 \quad (1 \leq j \leq k), 
\]
which implies $v_k = v_{k-1} = \cdots = v_0=0$ since $M$ is $f$-saturated.  
\end{proof}


\begin{theorem}
Let $f \in K[x]$ be a non-constant polynomial. 
Let $M = D_Xu$ be a left $D_X$-module generated by $u$ which is 
holonomic on $X_f := \{x \in X \mid f(x) \neq 0\}$. 
Then $M(u,f,\lambda)$ and $M(u,f,s)/tM(u,f,s)$ are holonomic $D_X$-modules 
for any $\lambda \in K$. 
\end{theorem}

\begin{proof}
Since $M(u,f,s) = \iota(M)(\iota(u),f,s)$, we may assume $M$ to be a nonzero 
holonomic $D_X$-module and $f$-saturated
replacing $M$ by $\iota(M)$. 
Since $N := M\otimes_{K[x]}B_{Z|Y}$ is holonomic, 
there exists a good Bernstein filtration $\{F_k(N)\}$ on $N$ and 
a polynomial $p(k)$ of degree $n+1$ such that 
$p(k) = \dim_K F_k(N)$ if $k$ is sufficiently large. 
Then $F_k(M(u,f,s)) := F_k(N) \cap M(u,f,s)$ is a filtration on 
$M(u,f,s)$ with respect to the weight vector $(1,\dots,1,2)$ 
for $(x,\partial_x,s)$.  

On the other hand, applying  a  well-known fact in commutative algebra 
(e.g., Theorem 4.4.3 in \cite{BH}) to the graded module, 
we can show that there exist a good filtration 
$\{G_k(M(u,f,s))\}$ on $M(u,f,s)$ with respect to the weight vector above, 
and two polynomials $q_1(k)$ and $q_2(k)$ of the same degree $d$ such that 
\[
\dim_K G_{2k}(M(u,f,s)) = q_1(2k), \quad
\dim_K G_{2k+1}(M(u,f,s)) = q_2(2k+1) 
\quad (\forall k \gg 0). 
\]
There exists $k_0\in\Z$ such that $G_k(M(u,f,s)) \subset F_{k+k_0}(M(u,f,s))$ 
for any $k \in \Z$. This implies that $d \leq n+1$. 

Set $N' = M(u,f,s)/tM(u,f,s)$ and 
\[
G_k(N') = G_k(M(u,f,s))/(tM(u,f,s) \cap G_k(M(u,f,s)).
\]
Then $\{G_k(N')\}$ constitutes a Bernstein filtration on the 
left $D_X$-module $N'$ (i.e., ignoring the action of $s$). 
Here note that we do not know at this stage whether $N'$ is 
finitely generated over $D_X$ or not. 

Since $t : M(u,f,s) \rightarrow M(u,f,s)$ 
is injective, we have
\begin{align*}
\dim_K G_k(N') &= \dim_K G_k(M(u,f,s)) - \dim_K (tM(u,f,s) \cap G_k(M(u,f,s))
\\&
\leq \dim_K G_k(M(u,f,s)) - \dim_K t^2G_{k-2}(M(u,f,s))
\\&
= \dim_K G_k(M(u,f,s)) - \dim_K G_{k-2}(M(u,f,s))
\\&
= \left\{ \begin{array}{ll}
q_1(k) - q_1(k-2) & \mbox{ if $k \gg 0$ is even)} \\
q_2(k) - q_2(k-2) & \mbox{ if $k \gg 0$ is odd)}
\end{array}\right.
\end{align*}
Since the degree of $q_i(k) - q_i(k-2)$ ($i=1,2$) is $d-1 \leq n$, 
this inequality implies that an arbitrary finitely generated $D_X$-submodule 
of $N'$ is holonomic and its multiplicity is bounded in terms of 
the leading coefficients of $q_1(k)$ and $q_2(k)$. 
Hence we conclude that $N'$ itself is holonomic. 

We can prove the holonomicity of  $M(u,f,\lambda)$, which is generated by 
$(u\otimes f^s)|_{s=\lambda}$,  in the same way 
replacing $t^2$ by $s-\lambda$ 
since $s-\lambda$ is an injective endomorphism of $M(u,f,s)$. 
\end{proof}

The first statement of the following theorem is given in 6.5  of \cite{KashiwaraAlg} 
for the case $M = K[x]$ and $u=1$. 

\begin{theorem}\label{th:phi}
Let  $M = D_Xu$ be a $D_X$-module generated by $u$ and $f\in K[x]$ be 
a non-constant polynomial. 
Assume that the $b$-function $b_{u,f}(s)$ exists. 
Let $\lambda$ be an arbitrary element of $K$ and 
define the $D_X$-homomorphism 
$\varphi_\lambda : M(u,f,\lambda+1) \rightarrow M(u,f,\lambda)$ by 
$\varphi_\lambda(P((u\otimes f^s)|_{s=\lambda+1})
= P(f(u\otimes f^s)|_{s=\lambda})$ for $P \in D_X$. 
\begin{enumerate}
\item
The following conditions are equivalent:
\begin{enumerate}
\item
$b_{u,f}(\lambda) \neq 0$
\item
$\varphi_\lambda : M(u,f,\lambda+1) \rightarrow M(u,f,\lambda)$ is 
an isomorphism. 
\end{enumerate}
\item Assume that $M$ is holonomic on $X_f$. Then one has 
\[
\mult M(u,f,\lambda + k) = \mult M(u,f,\lambda), 
\qquad
\length M(u,f,\lambda + k) = \length M(u,f,\lambda)
\]
for any $\lambda \in K$ and any integer $k$. 
In particular, one has
\[
\mult M[f^{-1}] = \mult M(u,f,k),
\qquad
\length M[f^{-1}] = \length M(u,f,k)
\]
for any integer $k$. 
\end{enumerate}
\end{theorem}

\begin{proof}
There exists a commutative diagram
\[
\xymatrix{
&&& 0 \ar[d] & \\
& 0 \ar[d] & 0 \ar[d] &  \Ksc_0 \ar[d] & \\ 
0 \ar[r] & M(u,f,s)\ar[r]^{t} \ar[d]^{s-\lambda-1} &  
M(u,f,s)\ar[r] \ar[d]^{s-\lambda} & M(u,f,s)/tM(u,f,s)\ar[r] 
\ar[d]^{s-\lambda}& 0
\\
0 \ar[r] & M(u,f,s)\ar[r]^{t} \ar[d] &  
M(u,f,s)\ar[r] \ar[d] & M(u,f,s)/tM(u,f,s)\ar[r] \ar[d] & 0
\\
 & M(u,f,\lambda+1) \ar[r]^{\varphi_\lambda} \ar[d] &  
M(u,f,\lambda) \ar[r] \ar[d] & \Ksc_1 \ar[d]
\\
& 0 & 0 & 0 &
}\]
of left $D_X$-modules, where the three vertical sequences 
and the upper two horizontal sequences are exact. 
Hence by the snake lemma we obtain an exact sequence 
\begin{equation}\label{eq:exact}
0 \longrightarrow \Ksc_0 \longrightarrow M(u,f,\lambda+1)
\stackrel{\varphi_\lambda}{\longrightarrow} M(u,f,\lambda) 
\longrightarrow \Ksc_1 \longrightarrow 0  
\end{equation}
of left $D_X$-modules. 

(1)
Assume $b_{u,f}(\lambda) \neq 0$. Then there exist 
$a(s),c(s) \in K[s]$ such that $a(s)(s-\lambda) + c(s)b_{u,f}(s)$ = 1. 
Hence for any $Q(s) \in D_X[s]$,
\[
Q(s)(u\otimes f^s) = Q(s)c(s)b_{u,f}(s)(u\otimes f^s) 
 + (s-\lambda)Q(s)a(s)(u\otimes f^s) 
\]
belongs to $tM(u,f,s) + (s-\lambda)M(u,f,s)$.   
If $(s-\lambda)Q(s)(u\otimes f^s)$ belongs to $tM(u,f,s)$, then 
\[
Q(s)(u\otimes f^s) = a(s)(s-\lambda)Q(s)(u\otimes f^s) 
  + Q(s)c(s)b_{u,f}(s)(u\otimes f^s)
\]
belongs to $tM(u,f,s)$. Hence $s-\lambda$ is an automorphism of 
$M(u,f,s)/tM(u,f,s)$. 

Conversely, assume that $s-\lambda$ is an automorphism of 
$M(u,f,s)/tM(u,f,s)$. Then the minimal polynomial $b_{u,f}(s)$ of $s$ 
on this module cannot be a multiple of $s-\lambda$. 
Summing up we have shown that $b_{u,f}(\lambda) \neq 0$ if and only if 
$\Ksc_0 = \Ksc_1 = 0$. 
In view of the exact sequence (\ref{eq:exact}), this is also equivalent 
to $\varphi_\lambda$ being an isomorphism.

(2)
We may assume that $M$ is a holonomic $D_X$-module and that 
$M$ is $f$-saturated replacing  $M$ by $\iota(M)$. 
Since $M(u,f,s)/tM(u,f,s)$ is holonomic, 
the length (and the multiplicity) of $\Ksc_0$ and the length (and 
the multiplicity respectively)  of $\Ksc_1$ are the same 
in view of the rightmost vertical exact sequence.  
Combined with this fact the exact sequence (\ref{eq:exact}) 
proves the statement (2). 
\end{proof}

This theorem provides us with an algorithm to compute the multiplicity of 
$M[f^{-1}]$ without any information on $b_{u,f}(s)$; thus we have only
to compute a Gr\"obner basis, e.g., of $M(u,f,0)$ with respect to 
a term order compatible with the Bernstein filtration. 

\begin{theorem}\label{th:rho}
The homomorphism $\tilde\rho_\lambda : 
M(u,f,\lambda) \rightarrow D_X(u\otimes f^\lambda)$ 
is an isomorphism if and only if $b_{u,f}(\lambda - k) \neq 0$ 
for any positive integer $k$. 
\end{theorem}

\begin{proof}
If $b_{u,f}(\lambda - k) \neq 0$ for any positive integer $k$, 
then $\tilde\rho_\lambda$ is an isomorphism by virtue of Proposition 
\ref{prop:fs1}. 
Now suppose $b_{u,f}(\lambda - k) = 0$ holds for some positive 
integer $k$ and let $k_0$ be the maximum among such $k$. 
Then Proposition \ref{prop:fs1} and Lemma \ref{lemma:fs2} imply that 
$\tilde\rho_{\lambda-k_0}$ is an isomorphism and that  
$D_X(u\otimes f^{\lambda-k_0+1}) \subsetneqq D_X(u\otimes f^{\lambda - k_0})$.
Hence by (2) of Theorem \ref{th:phi} we have
\begin{align*}
\length M(f,u,\lambda) 
&= \length M(f,u,\lambda - k_0) 
= \length D_X(u\otimes f^{\lambda - k_0})
\\&
> \length D_X(u\otimes f^{\lambda - k_0+1})
\geq \length D_X(u\otimes f^{\lambda}).
\end{align*}
Thus $\rho_\lambda$ is not an isomorphism. 
\end{proof}

\begin{corollary}
$M(u,f,\lambda)$ is $f$-saturated if and only if 
$b_{u,f}(\lambda - k) \neq 0$ for any positive integer $k$. 
In general, $\iota(M(u,f,\lambda))$ is isomorphic to 
$D_X(u\otimes f^\lambda)$. 
\end{corollary}

\begin{proof}
We may assume $M$ to be $f$-saturated. 
First note that $M\otimes_{K[x]}K[x,f^{-1}]f^\lambda$ is $f$-saturated 
for any $\lambda \in K$ since it is isomorphic to $M[f^{-1}]$ as 
$K[x]$-module. Hence $M(f,u,\lambda) \cong D_X(u\otimes f^\lambda)$ 
is also $f$-saturated under the assumption on $b_{u,f}(s)$. 

Now assume $b_{u,f}(\lambda - k) = 0$ for some positive integer $k$. 
Then $\tilde\rho_\lambda$ is not injective. Thus there exists $P \in D_X$ 
such that $P((u\otimes 1)|_{s=\lambda}) \neq 0$ but 
$P(u\otimes u^\lambda) = 0$. The latter equality means that there exist 
$Q(s) \in D_X[s]$ and $m \in \N$ such that 
\[
P(u\otimes f^s) = Q(s)(u\otimes f^{s-m}) \mbox{ in $M\otimes_{K[x]}\Lsc$}, 
\quad Q(\lambda)(u\otimes f^{\lambda-m}) = 0 
\mbox{ in $M\otimes_{K[x]}K[x,f^{-1}]f^\lambda$}.  
\]
Take a sufficiently large $l \in \N$ so that $f^lQ(s)f^{-m}$ belongs to 
$D_X[s]$. Then we have
\[
f^lP(u\otimes f^s) = f^lQ(s)f^{-m}(u\otimes f^s), 
\qquad
f^lQ(\lambda)f^{-m}(u\otimes f^\lambda) = 0.
\]
This means $f^lP((u\otimes f^s)|_{s=\lambda}) = 0$. 
Hence $M(u,f,\lambda)$ is not $f$-saturated. 
The last statement also follows from this argument.  
\end{proof}


\begin{example}\rm
Set $n=2$ and write $x_1 = x$, $x_2 = y$. 
Let $u$ be the residue class of $1$ in $M = D_X/I$ with $I$ 
being the left ideal of $D_X$ generated by two operators
\begin{align*}
P_1 &= x(1-x)\dx^2 + y(1-x)\dx\dy, 
\qquad
P_2 = y(1-y)\dy^2 + x(1-y)\dx\dy. 
\end{align*}
This is Appell's hypergeometric system $F_1$ with all parameters equal to zero. 
The singular locus of $M$ is a line arrangement defined by 
$f := x(x-1)y(y-1)(x-y)$. Let $\iota : M \rightarrow M[f^{-1}]$ be the 
canonical homomorphism. Then $M[f^{-1}]$ is generated by $f^{-2}\iota(u)$ 
and $\iota(M)$ is given by
\[
\iota(M) = D_X\iota(u) = D_X/(D_X\dx\dy + 
D_X((1-x)\dx^2 - \dx) + D_X((1-y)\dy^2 - \dy)). 
\]
The $b$-function with respect to $u$ and $f$ is
\[
b_{u,f}(s) = (s+1)^3(s+2)^2\bigl(s + \frac23\bigr)^2
  \bigl(s + \frac43\bigr)^2\bigl(s + \frac53\bigr). 
\]
As to multiplicities we have 
\[
\mult M = 10,\quad \mult \iota(M) = 5,\quad \mult H^0_{(f)}(M) = 5, \quad
\mult M[f^{-1}] = 36. 
\]
It follows that $\mult H^1_{(f)}(M) = 31$. 
By the way, the multiplicity of $K[x,f^{-1}]$ is $12$. 
\end{example}

\begin{example}\rm
Set $X = K^4$ and let $M_A(\beta_1,\beta_2)$ be the $A$-hypergeometric system associated with 
the matrix
$A = \begin{pmatrix} 1 & 1 & 1 & 1 \\ 0 & 1 & 3 & 4 \end{pmatrix}$, 
which is taken from Example 4.3.9 of \cite{SST}. 
More concretely, $M_A(\beta_1,\beta_2) = D_X/H_A(\beta_1,\beta_2)$ with the left ideal 
$H_A(\beta_1,\beta_2)$ generated by operators
\begin{align*}&
x_1\partial_1 + x_2\partial_2 + x_3\partial_3 + x_4\partial_4 \beta_1, \quad
x_2\partial_2 + 3x_3\partial_3 + 4x_4\partial_4,
\\&
\partial_2\partial_4^2-\partial_3^3,\quad
\partial_1\partial_4-\partial_2\partial_3,\quad
\partial_2^2\partial_4-\partial_1\partial_3^2,\quad
\partial_1^2\partial_3-\partial_2^3
\end{align*}
with parameters $\beta_1,\beta_2 \in K$.
We have
\[
\mult M_A(0,0) = 16, \quad \mult M_A(1,2) = 17. 
\]
The singular locus of $M_A(0,0)$ and that of $M_A(1,2)$ is given by
\[
g(x) := 
x_1x_4(256x_1^3x_4^3+(-192x_1^2x_2x_3-27x_2^4)x_4^2-6x_1x_2^2x_3^2x_4-27x_1^2x_3^4-4x_2^3x_3^3) = 0. 
\]
The $b$-functions of $u :=\overline 1 \in M_A(1,2)$ and 
$x_1$, $x_4$ are 
\[
b_{u,x_1}(s) = 
b_{u,x_4}(s) = s(s+1)(s+2), 
\]
while the $b$-functions of $v :=\overline 1 \in M_A(0,0)$ and 
$x_1$, $x_2$ are both $(s+1)^2$. 
We can verify by the algorithm that  
$M_A(0,0)$ and $M_A(1,2)$ are $x_1$- and $x_4$-saturated. 
The computation of the localization with respect to $g$ is intractable.  
We do not know if $M_A(\beta_1,\beta_2)$ is $g$-saturated or not.  
We conjecture that the multiplicity of $M_A(\beta_1,\beta_2)$ is $16$ 
for all $(\beta_1,\beta_2)$ except $(1,2)$, as is the case with the holonomic rank 
(see \cite{SST}). 
\end{example}

\end{document}